\documentclass[12pt]{amsart}
\usepackage{geometry}
\usepackage{graphicx}
\usepackage{amssymb}
\usepackage{epstopdf}
\usepackage{amsmath,amscd}
\usepackage{amsthm}
\usepackage{enumerate}
\usepackage{url,verbatim,soul}
\usepackage{subfig}
\usepackage{enumitem}

\RequirePackage[colorlinks,citecolor=blue,urlcolor=blue]{hyperref}
\usepackage{breakurl}
\theoremstyle{plain}
\newtheorem{theorem}{Theorem}
\newtheorem{definition}[theorem]{Definition}
\newtheorem{lemma}[theorem]{Lemma}
\newtheorem{proposition}[theorem]{Proposition}

\newtheorem{remark}[theorem]{Remark}

\newcommand\es{\varnothing}

\newcommand\ol{\overline}

\newcommand \sA{\mathcal{A}}

\newcommand\sQ{{\mathcal Q}}
\newcommand\sH{{\mathcal H}}

\newcommand\RR{{\mathbb R}}
\newcommand\ZZ{{\mathbb Z}}

\newcommand\PP{{\mathbb P}}

\newcommand\tri{\triangle}

\newcommand\De{\Delta}
\newcommand\qq{\qquad}
\newcommand\q{\quad}
\newcommand\resp{respectively}

\newcommand\oo{\infty}
\newcommand\sG{{\mathcal G}}

\newcommand\sC{{\mathcal C}}

\newcommand\sN{{\mathcal N}}

\newcommand\sM{{\mathcal M}}

\newcommand\sK{{\mathcal K}}

\newcommand\de{\delta}

\newcommand\pc{p_{\text{\rm c}}}

\newcommand\pu{p_{\text{\rm u}}}

\newcommand\rev[1]{\textcolor{red}{#1}}

\newcommand\Gd{\what G}

\renewcommand\ell{l}

\newcommand\pd{\partial}

\newcommand\sm{\setminus}

\newcommand\lam{\lambda}

\newcounter{mycount}\newcounter{mycount2}\newcounter{mycount3}
\newenvironment{romlist}{\begin{list}{\rm(\roman{mycount2})}%
   {\usecounter{mycount2}\labelwidth=1cm\itemsep 0pt}}{\end{list}}
\newenvironment{numlist}{\begin{list}{\arabic{mycount3}.}%
   {\usecounter{mycount3}\labelwidth=1cm\itemsep 0pt}}{\end{list}}
\newenvironment{letlist}{\begin{list}{\rm(\alph{mycount})}%
   {\usecounter{mycount}\labelwidth=1cm\itemsep 0pt}}{\end{list}}

\newcounter{newcount1}

\newcommand\diam{{\text{\rm diam}}}

\newcommand\what{\widehat}

\newcommand\nst{non-self-touching}
\newcommand\nt{non-touching}
\numberwithin{equation}{section}
\numberwithin{theorem}{section}
\numberwithin{figure}{section}

\newcommand\dinst{$2\oo$-nst path}

\newcommand{\hsim}{\mathbin{\what\sim}}
\newcommand{\nhsim}{\mathbin{\what\nsim}}

\newcommand\Gm{G_*}

\newcommand\ins{\text{\rm int\hskip.7pt}}
\newcommand\NST{\text{\rm NST}}
\newcommand\ST{\text{\rm ST}}
\newcommand\GD{G_\De}
\newcommand\TD{T_\De}
\newcommand\nuf{\what\nu}
\newcommand\GDd{\what G_\De}
\newcommand\VDd{\what V_\De}
\newcommand\EDd{\what E_\De}
\newcommand\Ext{\text{\rm Ext}}
\newcommand\Ed{\what E}
\newcommand\Vd{\what V}

\title{Non-self-touching paths in plane graphs}

\author{Geoffrey R.\ Grimmett}
\address{Centre for
Mathematical Sciences, Cambridge University, Wilberforce Road,
Cambridge CB3 0WB, UK} 

\email{g.r.grimmett@statslab.cam.ac.uk}
\urladdr{\url{http://www.statslab.cam.ac.uk/~grg/}}

\date{27 January 2024, 20 June 2025}

\begin{document}

\begin{abstract}
A path in a graph $G$ is called \nst\ if two vertices are neighbours in the path if and only if they are
neighbours in the graph. We investigate the existence of doubly infinite \nst\ paths in infinite plane graphs.

The \emph{matching graph} $\Gm$ of an infinite plane graph
$G$ is obtained by adding all diagonals to all faces,
and it plays an important role in the theory of site percolation on $G$. The main result of this paper
is a necessary and sufficient condition on $G$ for the existence of a doubly infinite \nst\ path in 
$\Gm$ that traverses some diagonal. This is a key step in proving, for quasi-transitive $G$, that the
critical points of site percolation on $G$ and $\Gm$ satisfy the strict inequality 
$\pc(\Gm) < \pc(G)$, and it complements the earlier result of Grimmett and Li 
(Random Struct.\ Alg.\ 65 (2024) 832--856),
proved by different methods,
concerning the case of transitive graphs. Furthermore it implies, for quasi-transitive graphs, that
$\pu(G) + \pc(G) \ge 1$, with equality if and only if the graph $\GD$, obtained from $G$
by emptying all separating triangles, is a triangulation. Here,
$\pu$ is the critical probability for the existence of a unique infinite open cluster.
\end{abstract}
\maketitle

\section{Background and main theorem}\label{sec:back}

Some basic facts are presented concerning the existence
in an infinite planar graph $G$ of a certain type of doubly infinite path,
namely a path $\pi$ with the property that two vertices of $\pi$ are neighbours in $G$ if and only 
if they are consecutive in $\pi$. Such paths arise naturally in the theory of site percolation. 

The graphs considered here are assumed to belong to the set
$\sG$ of countably infinite, locally finite, $2$-connected, simple, plane graphs, embedded in 
the plane $\RR^2$ without accumulation points, and moreover such that 
all faces have finite diameter. 
A doubly infinite path 
$\pi=(\pi_i: -\oo < i < \oo)$ of a graph is
called \emph{\nst} if it has the property that $\pi_i\sim\pi_j$ if and only if $|i-j|=1$.
The expression \lq doubly infinite \nst\ path' is abbreviated henceforth to \emph{\dinst}.

\begin{remark}\label{rem:-2}
There appears to be no standard expression for the path-property of being \nst, and we 
adopt this expression for consistency with early work \cite{GL-match}. Possible alternatives
include \lq chordless path' and \lq induced path'. The theory of \nst\ paths
is a matter of potential intrinsic interest in graph theory.
\end{remark}

Which graphs possess a \dinst? We do not have a complete answer to this, but certain cases are described in 
Section \ref{sec:exist}. For example, every $4$-connected $G\in\sG$ has a \dinst, and
every graph $G\in\sG$, embedded in $\RR^2$ in such way that its faces have
uniformly bounded diameter, has a \dinst. 

The \emph{matching graph} $\Gm$ of $G\in\sG$ is obtained from $G$ by adding all diagonals
to all non-triangular  faces (see Figure \ref{fig:mg}); the word \emph{diagonal} shall always mean such
an edge of $\Gm$. 
The principal purpose of this paper is to prove a property of the pair $(G,\Gm)$ of graphs.
Evidently, $\Gm=G$ if and only if $G$ is a triangulation. Note that, while $G$ is planar, 
its matching graph $\Gm$ is planar if and only if $G$ is a triangulation.

The following graph property is important in the theory of site percolation
(see Section \ref{sec:perc}).
 
\begin{figure}
\centerline{\includegraphics[width=0.5\textwidth]{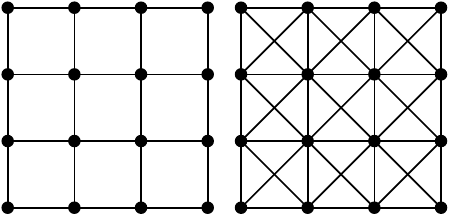}}
\caption{The square lattice $\ZZ^2$ and its matching graph.}\label{fig:mg}
\end{figure}

\begin{definition}
The graph $G\in\sG$ is said to have property $\Pi$ if $\Gm$ has a \dinst\ that includes some
diagonal of some non-triangular face of $G$.
\end{definition}

No triangulation can have property $\Pi$ since a triangulation has no diagonals. 

We call a $3$-cycle $C$ of a connected plane graph a \emph{separating triangle} 
if the bounded component of $\RR^2\sm C$ (termed the \emph{interior} of $C$)
\rev{intersects} one or more edges and/or vertices.
If $C$ is a separating triangle of $G\in\sG$, then 
no \dinst\ may intersect this interior. Thus the interiors of separating triangles may
be removed without changing the property of having a \dinst.
For $G\in\sG$, we write $\GD$ for the subgraph of $G$ obtained by deleting
any vertex/edge lying in the interior of any $3$-cycle of $G$. We shall normally assume that $\GD\in\sG$,
thereby eliminating the possibility that $G$ has an infinite nested sequence of $3$-cycles.
A graph $G\in\sG$ is said to be \emph{$\tri$-empty} if it contains no separating triangle. 

We prove the straightforward fact (in Theorem \ref{thm:A}(b))
that a triangulation $T$ has a \dinst\ if $\TD\in\sG$.
An example of a graph $G\in\sG$ with a separating triangle but without property $\Pi$ is 
given in Figure \ref{fig:nopi2}.
 
 \begin{figure}
\centerline{\includegraphics[width=0.3\textwidth]{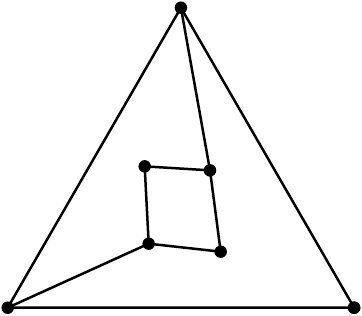}}
\caption{The graph $G\in\sG$ is obtained from the usual triangular lattice by replacing 
one of more  fundamental triangles with a
copy of the above. The ensuing graph cannot have property $\Pi$ since 
no \dinst\ may penetrate any fundamental triangle.}\label{fig:nopi2}
\end{figure}

Here is the main theorem. Its application to percolation theory is outlined in Section \ref{sec:perc}.

\begin{theorem}\label{thm:0}
Let $G\in\sG$ satisfy $\GD\in\sG$, and assume $\GD$ is not a triangulation. 
If  $\Gm$ has a \dinst, then $G$ has property $\Pi$.
\end{theorem}

The basic idea of the proof of Theorem \ref{thm:0} is as follows. 
Since  $\GD$ is not a triangulation, it has some face $F$ with four or more edges in its boundary.
Assume $\Gm$ has a \dinst\ $\nu$.  The target is to show that one can make local changes to $\nu$
in order to obtain a \dinst\ $\ol\nu$ that uses some diagonal of $F$. There are some  difficulties in
achieving this, and indeed a lesser target is
achieved that is sufficient for the theorem. The construction is facilitated by working 
not with $G$ directly but with the triangulation $\Gd$
(the \lq facial graph' of Section \ref{sec:face}) obtained from $G$
by adding a site to each non-triangular face, and fully connecting this site to the
boundary cycle. It is then necessary to understand the relationship between \dinst s of $\Gm$ and 
\dinst s of $\Gd$.

One of the reasons for working with $\Gd$ is that, as a triangulation, one may show the existence of
an infinite, nested sequence of cycles with $F$ in their common interior. This permits an iterative
approach to the construction of $\ol\nu$.

Here is a summary of the contents of this article. 
The application of Theorem \ref{thm:0} to percolation is presented in 
Section \ref{sec:perc}. After a section on notation, and the methodological Section \ref{sec:tech},
the principal graph-theoretic Proposition \ref{lem:main} appears in Section \ref{sec:princ}.
The cycle structure of plane graphs is explored in Section \ref{sec:step0}, which ends with the proof
of Theorem \ref{thm:0} (using Proposition \ref{lem:main}). Sections \ref{sec:step2}
and \ref{sec:step3} are devoted to the proof of Proposition \ref{lem:main}.

The proof of Proposition \ref{lem:main} is a somewhat complicated graph-theoretic analysis of
a number of possible cases. 
It is tempting to hope for a neater and more  appetising
proof of Theorem \ref{thm:0}.

\section{Application to site percolation}\label{sec:perc}

The percolation process is a prominent model for connectivity in a random medium. 
The model has emerged as central to the mathematical and physical theories of phase transition,
and its theory is ramified and
complex.  
Percolation comes in two flavours, bond and site, and it is site percolation that is relevant here.
See \cite{GP} for an account of the standard theory of percolation. 

Let $G=(V,E)$ be an infinite connected graph, and let $p\in[0,1]$. 
Each vertex (or \lq site')
$v\in V$ is coloured \emph{black} with probability $p$ and \emph{white} otherwise, different vertices receiving
independent colours. We write $\PP_p$ for the corresponding probability measure. 
We choose some vertex, called the \emph{origin}, and write
$I$ for the event that the origin is the endpoint of some infinite black path. With
$\theta(p)= \PP_p(I)$, there exists a \emph{critical probability} $\pc=\pc(G)\in [0,1]$ such that
\begin{equation}
\theta(p) 
\begin{cases} =0 &\text{if } p<\pc(G),\\
>0 &\text{if } p>\pc(G).
\end{cases}
\end{equation}
The value of $\pc(G)$ is independent of the choice of origin.

The study of weak and strict inequalities for critical probabilities has a long history (see, for example, \cite{jmh61}
and \cite[Sect.\ 10]{K82}).
A general method for proving strict inequalities for critical probabilities, and more generally for critical points of
interacting systems, was described in \cite{AG}. One assumption for a naive application of this method is the
quasi-transitivity of
the underlying graph $G=(V,E)$. 
Recall that $G$ is \emph{quasi-transitive} if its automorphism 
group acts on $V$ with only finitely many orbits.

Since $G$ is a subgraph of $\Gm$, it is elementary that $\pc(\Gm) \le \pc(G)$.
Strict inequality is harder to prove. The following was proved in \cite{GL-match}.

\begin{theorem}[\mbox{\cite[Thm 1.2]{GL-match}}]\label{maintheorem}
Let $G\in\sG$ be quasi-transitive. Then $\pc(G_*)<\pc(G)$ if
and only if $G$ has property $\Pi$.
\end{theorem}

Using Theorems \ref{maintheorem} and \ref{thm:0}, one obtains the following application  to percolation
of the results of this article.

\begin{theorem}\label{thm:main2}
Let $G\in\sG$ be quasi-transitive.
The strict inequality $\pc(\Gm)<\pc(G)$ holds if only if $\GD$ is not a triangulation.
\end{theorem}

This extends 
the earlier result of \cite[Thm 1.4]{GL-match} which was restricted to transitive graphs, 
for which the proof is different and less complicated.

\begin{proof}[Proof of Theorem \ref{thm:main2} using Theorem \ref{thm:0}]
Since $G$ is assumed quasi-transitive, we have that $\GD$ is quasi-transitive
and belongs to $\sG$
(this is easily seen, and a formal statement with proof appears at Theorem \ref{thm:A}(e)). 
Every infinite path of $G$ contains an infinite path of $\GD$, and conversely an infinite path of
$\GD$ is an infinite path of $G$. Therefore, $G$ and $\GD$ 
(\resp, their two matching graphs) have equal critical probabilities. 

If $\GD$ is a triangulation, then its matching graph is also $\GD$, so that their
critical probabilities are equal.
Assume that $\GD$ is not a triangulation.
By Theorems \ref{thm:0} and \ref{maintheorem}, it suffices to show that $\Gm$ has a \dinst.
This is included in \cite[Lemma 4.3(a)]{GL21a}, and is given explicitly in Theorem \ref{thm:A}(d).
\end{proof}

Non-self-touching paths were introduced in \cite{AG} where they were called \lq stiff paths` (see
also \cite{BBR,GL-match} and \cite[p.\ 66]{GP}).

Suppose $H$ is a connected, quasi-transitive graph. Let $N$ be the number of infinite black clusters of site percolation on $H$.
It was proved in \cite{HP, SR99} that there exists $\pu(H)\in [0,1]$ such that
\begin{equation*}
\PP_p(N=1) = \begin{cases}
0 &\text{if } p < \pu(H),\\
1 &\text{if } p > \pu(H).
\end{cases}
\end{equation*}
Evidently, $\pc (H)\le \pu(H)$.
Let $G\in\sG$ be quasi-transitive. It is known that $\pu(G)+\pc(\Gm)=1$ 
(see \cite[Thm 1.1]{GL21a}), and it follows
by Theorem \ref{thm:main2} that $\pu(G)+\pc(G)\ge 1$ with equality
if and only if $\GD$ is a triangulation. 

\section{Notation}\label{sec:not}

A graph is denoted $G=(V,E)$ where $V$ is the vertex-set and $E$ the edge-set. Graphs 
considered here are mostly assumed to
be countable (that is, finite or countably infinite),  and simple (in that they have neither loops nor parallel edges);
a possible exception to the last arises in the case of matching graphs, which may contain
pairs of parallel diagonals created in abutting faces.
An edge between vertices $u$, $v$ is denoted $\langle u,v\rangle$; if this edge exists, we say that
$u$ and $v$ are \emph{adjacent} and write $u \sim v$. The edge $\langle u,v\rangle$
is said to be \emph{incident} to its endvertices.
The \emph{degree} of a vertex is the number of its incident edges, and $G$ is \emph{locally finite}
if all degrees are finite. Given $A,B\subseteq V$, $A$ is said to be 
\emph{adjacent} to $B$, written $A \sim B$, if there 
exist $a\in A$ and $b\in B$ such that $a \sim b$.

A \emph{walk} in $G$ is an alternating sequence $w=(\dots,w_0,e_0,w_1,e_1,\dots)$ where $w_i\in V$ and 
$e_i=\langle w_i,w_{i+1}\rangle\in E$ for all $i$;
if $G$ is simple, the edges $e_i$ may be omitted from the definition. The walk $w$ is a \emph{path} if the $w_i$ are distinct.
The path $w$ is \emph{\nst} if $w_i\sim w_j$ if and only if $|i-j|=1$.  A path $w$ is called a \emph{\dinst}
if it is doubly infinite and \nst; we denote by $\NST(G)$ the set of all \dinst s of $G$.
The graph-distance $d_G(u,v)$ between vertices $u$, $v$ is the minimal
number of edges in paths from $u$ to $v$; for $A,B\subseteq V$; we set
$d_G(A,B)=\min\{d_G(a,b): a\in A,\, b\in B\}$.
Two walks $\pi=(\pi_i)$, $\nu=(\nu_j)$ are said to be
\emph{non-touching} if $d_G(\pi_i,\nu_j)\ge 2$ for every pair $i$, $j$.
A path from $u$ to $v$ is called a \emph{geodesic} if it has exactly  $d_G(u,v)$ edges. 
We note that a finite path is \nst\ if it is a geodesic; a similar statement holds for 
infinite paths. 

A \emph{cycle} of $G$ is a finite walk of the form $w=(w_0,e_0,w_1,\dots,w_{n})$
such that $w_0=w_{n}$ and the sub-walk $(w_0,e_0,w_1,\dots, w_{n-1})$ is a path. Such a cycle
has \emph{length} $n$ and is called an $n$-\emph{cycle}. The set of cycles of $G$
is denoted $\sC(G)$.

Let $k\ge 1$. An infinite graph $G$ is called \emph{$k$-connected} if, for all $v\in V$, there exist at least $k$
infinite paths starting from $v$ that are pairwise vertex-disjoint (except for their common starting point $v$).
By Menger's theorem (see, for example, \cite[Thm 1.1]{ACM}), $G$ is $k$-connected if and only if, for all $v\in V$, there exists no set $A\subseteq V\sm\{v\}$
of cardinality strictly less that $k$ whose removal leaves $v$ in a finite subgraph of $G$.
For further discussion and references, see \cite[Sect.\ 1]{ACM} and \cite{Dirac63,halin70}.

A graph $G=(V,E)$ is \emph{planar} if it may be drawn in the plane in such a way that 
edges cross only at vertices. An embedded planar graph is called \emph{plane}.
A point $x\in\RR^2$ is called a \emph{vertex accumulation point} of $G$ 
if it is an accumulation point of $V$, and an 
\emph{edge-accumulation point} if every neighbourhood
of $x$ intersects some edge not incident with $x$.
We shall consider only plane graphs with neither  vertex- nor edge-accumulation points.
The number of \emph{ends} of a graph is
the supremum of the number of infinite components obtained by deletion of finite sets of vertices.

A \emph{face} of a one-ended, plane graph $G=(V,E)$ is a connected component of $\RR^2\sm G$.
By \cite[Thm 3]{Kr},  if $G$ is $2$-connected,
the boundary of every face $F$ is a cycle of $G$, denoted $\pd F$. The \emph{size} of
the face $F$ is the number of edges in $\pd F$, and its (Euclidean) \emph{diameter}
is defined as 
\begin{equation*}
\diam(F)=\sup\{|x-y|: x,y \in F\}
\end{equation*}
where $|\cdot|$ denotes Euclidean distance.
Let $C$ be a cycle of $G$, and write $\ins(C)$ for the  (open) bounded component of
$\RR^2\sm C$, and $\ol C =  C\cup \ins(C)$. We write $\ins(C)$ also for the subgraph of $G$ 
obtained by deleting all vertices not belonging to $\ins(C)$. A cycle $C$ is called
\emph{facial} if it is the boundary of some face.

We denote by $\sG$ the set of countably infinite, $2$-connected, locally finite, 
simple, plane graphs, embedded in 
the plane without vertex/edge-accumulation points, such that all faces have finite diameter
(whence, in particular, such $G$ are one-ended). 

We call a $3$-cycle of $G$  a \emph{separating triangle} if $\ins(C)$ intersects
one or more edges and/or vertices of $G$.
For $G\in\sG$, we write $\GD$ for the subgraph of $G$ obtained by deleting
any vertex/edge lying in the interior of any separating triangle of $G$. Thus $\GD$ has no separating
triangle, and we say that $\GD$ is $\tri$-\emph{empty}. We shall speak of $\GD$ as 
being obtained from $G$ by \lq emptying the separating triangles'.
Since  a \dinst\ of $G$ intersects the interior of no separating triangle, we have that
\begin{equation}\label{eq:noint}
\NST(G)=\NST(\GD).
\end{equation}

The one-ended, plane graph $G$ is a \emph{triangulation} if every face is bounded by a $3$-cycle. 
Let $u,v\in V$ be such that $u\nsim v$ but there exists some face
$F$ with $u,v\in \pd F$; we may choose to add to $F$ the further edge $\langle u,v \rangle$, and we call this
a \emph{diagonal} of $G$ (or of $\Gm$), denoted $\de(u,v)$. 

The \emph{matching graph} $\Gm$ of $G\in\sG$
is obtained from $G$ by adding all diagonals
to all non-triangular faces. See Figure \ref{fig:mg}, and note that $\Gm$ is not generally planar.
We shall work also with the so-called \lq facial graph' of $G$; see Section \ref{sec:face}.
The matching graph was introduced by Sykes and Essam \cite{SE63} in the context of
percolation theory. 

The reasons for the assumption of $2$-connectivity are as follows. Let $G$
be $1$-connected but not $2$-connected. Then there exist cutpoints $c$ such that $G\sm\{c\}$ has one or more finite components. Such components cannot be relevant to the occurrence or not of property $\Pi$ since
no \dinst\ (of either $G$ or $\Gm$)
may access them. Linked to this is the fact that site percolation on $G$ possesses 
an infinite cluster if and only  $G\sm\{c\}$ contains such a cluster. Moreover, as remarked above, 
$2$-connectivity is needed for the faces of $G$ to be bounded by cycles.

\begin{remark}\label{rem:planar}
We close this section with a note about the distinction between planar and plane graphs.
A planar graph $H$ is said to have property $\sN$ if it possesses a \dinst. Evidently $\sN$
is a graph property of $H$ which is independent of the choice of plane embedding.
The situation for matching graphs is potentially more complicated since the
diagonals of a plane graph depend on its facial structure and hence on its embedding.
If $H$ is $3$-connected, its embedding is unique
in the sense of the cellular-embedding theorem of \cite[p.\ 42]{Moh}; 
see also \cite[Thm 2.1]{GL21a}. Therefore, $\sN$ is a graph property in this case. 

The picture is more complicated if $H$ has connectivity $2$. 
Assume this, and in addition \emph{that $H$ is quasi-transitive}. Let $G\in\sG$ be a plane embedding of $H$. 
By the proof of Theorem 8.25 in \cite[Sect. 8.8]{LP}, there exists a $3$-connected 
plane graph $G'$ from which $G$ is obtained by adding certain \lq dangling loops'.
Since $G'$ is $3$-connected, by the cellular-embedding theorem its embedding is
unique as above,
so that every embedding of $H$ gives rise to the same $G'$. Furthermore,
one sees from the relationship between $G$ and $G'$ that $G$ has $\sN$ if and only if $G'$ has $\sN$.
In conclusion, for $2$-connected, \emph{quasi-transitive} planar graphs, property $\sN$ is a graph property
and is independent of the choice of plane embedding. We shall see in Theorem \ref{thm:A}(d) that
one such embedding, and hence all such embeddings, have property $\sN$.
\end{remark}

\section{Three techniques}\label{sec:tech}

\subsection{Oxbow removal}\label{sec:ox}

Paths can fail to be \nst\ through the existence of pairs of vertices that are not neighbours in the path 
but are neighbours in the graph. It is useful to have a method for extracting a \nst\ path from a path containing such vertex-pairs. The method in question was used in \cite{GL-match}, and is termed \emph{oxbow removal}.
We shall make use of the following extract from \cite[Lemma 4.1(b)]{GL-match}.

\begin{lemma}\label{lem:cred}
Let $H$ be a simple, plane graph embedded in $\RR^2$.
Let $\pi$ be a finite (\resp, infinite) path with endpoint $v$.
There exists a non-empty subset 
$\pi'$ of the vertex-set of $\pi$ that forms a finite (\resp, infinite)  \nst\ path of $H$ starting at $v$.
If $\pi$ is finite, 
then $\pi'$ may be chosen with the same endvertices as $\pi$.
\end{lemma}

The related process of \lq loop-erasure'
is familiar in graph theory and probability; see, for example, \cite[Sect.\ 2.2]{G-pgs}.
As noted in Section \ref{sec:not}, a geodesic is \nst.
By Lemma \ref{lem:cred}, every locally finite, infinite, connected, simple
 graph possesses a singly infinite \nst\ path.

\begin{proof}
Let $\pi=(v_0,v_1,v_2, \dots)$ be a path from $v=v_0$, either finite or infinite. 
We start at $v_0$ and move along $\pi$ in 
increasing order of vertex-index. 
Let $J$ be the least $j$ such that there exists $i\in\{0,1,\dots,j-2\}$ 
with $v_i\sim v_J$, and let $I$ be the earliest such $i$. 
We delete from $\pi$ the subpath $(\pi_{I+1},\dots,\pi_{J-1})$ (which is termed an \emph{oxbow}),
thus obtaining a new path $\pi_1$ starting at $v$.
If $\pi$ is finite then $\pi_1$ has the same endvertices as $\pi$.
This process is iterated until no oxbows remain. 
\end{proof}

\subsection{Existence of \dinst s}\label{sec:exist}

We present an elementary theorem concerning the existence of \dinst s. 
Recall the graph $\GD$, obtained from $G$ by emptying all $3$-cycles; see before Theorem \ref{thm:0}.

Here is some notation.
A face $F$ of $G\in\sG$ satisfying $0\notin \ol F$ is called $\zeta$-\emph{acute}
if there exists a sector $S$ of $\RR^2$ with vertex $0$ and angle $\zeta$ such that
$F\subseteq S$.

\begin{theorem}\label{thm:A}\mbox{\hfil}
\begin{letlist}
\item
Let $G$ be an infinite, connected, plane graph such that $\GD$ is $4$-connected. Then $G$ contains a \dinst.

\item
Every infinite, $\tri$-empty, triangulation $T$ contains a \dinst.

\item
Let $G\in\sG$. Suppose there exists $\zeta\in(0,\frac12\pi)$ such that $F$ is $\zeta$-acute
for all but finitely many faces $F$. Then $G$ and $\Gm$ have \dinst s.

\item
If $G\in\sG$ is quasi-transitive, then $G$  and $\Gm$ have \dinst s.

\item
If $G\in\sG$ is quasi-transitive, then $\GD\in\sG$ and $\GD$ is quasi-transitive.
\end{letlist}
\end{theorem}

The conditions of (a) and (c) are sufficient but evidently not necessary for the existence of a \dinst.
Instances of \nst\ paths are provided by geodesics, and the existence of infinite geodesics has been explored in
several articles including \cite{BIS,SMartin,Wat}. Figure \ref{fig:2conn} contains an illustration of
a $3$-connected $G\in\sG$ such that neither $G$ nor its matching graph has a \dinst.

\begin{figure}
\centerline{
\includegraphics[width=0.45\textwidth]{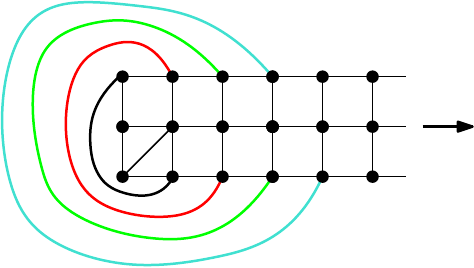}}
\caption{A $3$-connected graph $G\in\sG$ without separating triangles 
such that neither $G$ nor $\Gm$ has a \dinst. Each vertex in the upper horizontal line is joined
to the vertex one unit to its right in the lower line. A diagonal has been added to ensure the graph is 
truly $3$-connected.}\label{fig:2conn}
\end{figure}

\begin{proof}
(a) Let $G=(V,E)$ be as stated. By the $4$-connectedness of $\GD$,
for $v\in V$, there exist four infinite paths of $\GD$
from $v$ that are pairwise vertex-disjoint except for the point $v$. 
Label these $\pi_i$ in a clockwise manner,
and write $\pi_i^-=\pi_i\sm\{v\}$. 
Then $d_{\GD}(\pi_1^-,\pi_3^- ) \ge 2$.
For $i=1,3$, the path $\pi_i^-$ may be reduced by oxbow removal 
(see Lemma \ref{lem:cred}) to a 
singly infinite \nst\ path $\nu_i$ with the same endvertex as $\pi_i$.
The path $\nu:=\nu_1 \cup \{v\}\cup\nu_3$ contains the required
\dinst. On adding the contents of the original triangles back into $\GD$,
we see that
$\nu$ is a \dinst\ of $G$.

(b)
Let $T=(V,E)$ be as in the statement of the theorem.
Since $T$ is $\tri$-empty, it is $4$-connected (see, for example, \rev{\cite[Sect.\ 1]{BKK}} 
and \cite[p.\ 91]{KU}), and the claim
follows by part (a).

For the sake of completeness, we include a sketch proof of the $4$-connectedness of \rev{such} $T$. 
Suppose that $T$ is not $4$-connected. It is standard that, as a triangulation,
$T$ is $3$-connected.
Therefore, there exists $v\in V$ such that
the maximum number of infinite paths from $v$ that are pairwise vertex-disjoint (except at $v$) is exactly $3$.
By Menger's theorem, there exists a triple $A=\{a,b,c\}$ of distinct vertices (with $v\ne a,b,c$) such that every 
infinite path from $v$ intersects $A$, and $A$ is minimal with this property. 

Let $C$ be the (finite) connected component containing $v$ in the graph $T$ with $A$ deleted.
Since $A$ is a minimal cutset, there exist $a',b',c'\in C$ such that $\langle a,a'\rangle,  \langle b,b'\rangle,
\langle c,c'\rangle $ are edges of $T$. Since $T$ is a maximal triangulation
(in that no edge may be added to to $T$ without contradicting planarity), the edges 
$\langle a,b\rangle,  \langle b,c\rangle,
\langle c,a\rangle$ exist in $T$. That is, $A$ is a separating triangle. 
By assumption, $T$ has no separating triangle, and therefore $T$ is $4$-connected.

(c) We outline the proof, which is an adaptation of that of \cite[Lemma 4.3(a)]{GL-match}. 
Suppose the condition holds, and let $L_\theta$ denote the singly infinite
straight line from $0$ inclined at clockwise angle $\theta$ to the $x$-axis $X$. 
Let $S_+$ be the closed sector between $L_0$ and $L_\zeta$ (clockwise), and let $I_+$ be the property that
$G$ has some singly infinite path $\pi_+$ lying within $S_+$. If $I_+$ fails to hold, there 
exists a family $\sK$ of arcs of $S_+$ ($\subseteq\RR^2$), each with endpoints in $L_0$ and $L_\zeta$,
 such that 
(i) each $\kappa\in \sK$
intersects no edge of $G$, and (ii) the Euclidean distances $d(0,\kappa)$ are unbounded
as $\kappa$ ranges over $\sK$. Each $\kappa\in \sK$ lies in the interior of some face of $F$.
Since there exist only finitely many faces that intersect both $L_0$ and $L_\zeta$, 
the statement $I_+$ must hold. Write $\nu_+$ for a \nst\ path obtained from $\pi_+$ by oxbow removal
(see Lemma \ref{lem:cred}).

By a similar argument with $S_+$ replaced by $S_-:= -S$
(the sector bounded by $L_\pi$ and $L_{\pi+\zeta}$), 
$G$ has some singly infinite, \nst\ path $\nu_-$
lying in $S_-$. Since $\pi-\zeta > \frac12\pi>\zeta$, the set $\sA$ of faces that 
intersect both $S_-$ and $S_+$ is finite. 
Find a shortest path $\pi$ of $G$ that connects $\nu_+$ and $\nu_-$ and intersects no $F\in\sA$. The union
$\nu_-\cup\pi \cup \nu_+$ contains (after oxbow removal) a \dinst.

The same argument applies to the matching graph $\Gm$.

(d)
Let $H$ be quasi-transitive, and consider its plane embeddings
that belong to $\sG$. By Remark \ref{rem:planar}, either all or no plane embeddings
(\resp, their matching graphs) have \dinst s.
Since $H$ is quasi-transitive, it may be embedded in either the Euclidean or
hyperbolic plane (denoted $\sH$) in such a way that its edges are geodesics and its automorphisms extend 
to isometries of the plane (see \cite[Thm 8.25 and Sect.\ 8.8]{LP} and \cite[Thm 2.1]{GL21a});
let $G\in\sG$ be such an embedding of $H$ and consider 
for definiteness the hyperbolic case (in the model of the Poincar\'e disk --- 
see \cite{CFKP} for an account of hyperbolic geometry).
We note that, by the isometricity property, the diameters of faces of
$G$ are bounded uniformly above (in the hyperbolic metric).
The current claim is the content of \cite[Lemma 4.3(a)]{GL-match}. 

(e) There is a partial order $\le$ on the set $\ST(G)$ of separating triangles of $G=(V,E)$ given 
by $T_1\le T_2$ if $T_1\subseteq\ol {T_2}$. A triangle $T\in\ST(G)$ is \emph{maximal}
if it is maximal with respect to $\le$, and $\sM$ denotes the set of maximal triangles. 
Since $G$ is quasi-transitive without accumulation points,
for $T'\in\ST(G)$, there exists $T\in\sM$ with $T'\le T$. 
Since each $T\in\sM$ is a subgraph of $\GD$, and $G$ and $\GD$ agree off the union of the maximal triangles,
$\GD$ has only bounded faces. 

We show next that $\GD$ is $2$-connected. Let $v$ be a vertex of $\GD$. Since $v\in V$
and $G$ is $2$-connected,
there exist infinite paths $\pi_1$, $\pi_2$ of $G$ that are vertex-disjoint except at their common initial
vertex $v$. Let $T\in\sM$. If $\pi_i$ intersects $T$, we find the first (\resp, last) intersection point
$x$ (\resp, $y$), and we remove from $\pi_i$ the section of the path lying strictly between $x$ and $y$.
This results in a subpath $\pi_i(T)$ that does not intersect $\ins(T)$. The process is iterated
as $T$ ranges over $\sM$, and the outcome is an infinite subpath $\nu_i$ of $\pi_i$ lying in $\GD$. 
Therefore, $\GD$
is $2$-connected.

The quasi-transitivity of $\GD$ follows from that of $G$, and the claim is proved.
\end{proof}

\subsection{The facial graph}\label{sec:face}

\begin{figure}
\centerline{\includegraphics[width=0.6\textwidth]{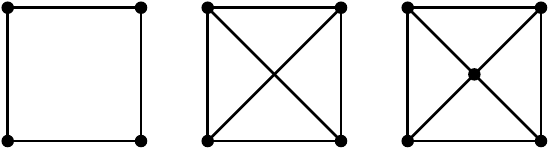}}
\caption{A square of the square lattice, its matching graph, and with its facial site added.}\label{fig:HK}
\end{figure}

Let $G\in\sG=(V,E)$, and let $\sQ$ be the set of all non-triangular faces of $G$. We shall work with the 
graph $\Gd=(\Vd,\Ed)$ obtained from $G$ by adding a new vertex within each face $F\in\sQ$,
and adding an edge from every vertex in the boundary $\pd F$ to this central vertex.
These new vertices are
called \emph{facial sites}, and the graph $\Gd$ is called the 
\emph{facial graph} of $G$. The facial site in the face $F$ is denoted $\phi(F)$. See 
\cite{GL-match},  \cite[Sect.\ 2.3]{K82}, \cite[Sect.\ 8.8]{LP}, and also Figure \ref{fig:HK}. 
If $\langle v,w\rangle$ is a diagonal of 
the matching graph $\Gm$, it
lies in some face $F$ of $G$ with four or more edges, 
and we write $\phi(v,w)=\phi_F(v,w) =\phi(F)$ for the corresponding facial site.

Of importance in this work
 is the graph $\GDd=(\VDd,\EDd)$, defined as the graph obtained by emptying the separating triangles of the 
facial graph $\Gd$. We note that $\GDd=(\Gd)_\De$ but generally
$\GDd \not= \what{(\GD)}$. The reason for this distinction lies in part (c) of the following lemma.
We recall the set $\NST(H)$ of \dinst s of a graph $H$. By \eqref{eq:noint} applied to
$\Gd$, we have that 
\begin{equation}\label{eq:noint2}
\NST(\Gd)=\NST(\GDd).
\end{equation}

\begin{lemma}\label{lem:4}
Let $G=(V,E)\in\sG$.
\begin{letlist}

\item Let $\nu\in\NST(\Gm)$, and let $F$ be a face of $G$. 
If $\nu\cap \pd F \ne\es$, then the intersection is 
exactly one of the following:  (i) a single vertex of $G$, (ii) a single edge of $G$, 
(iii) a single diagonal of $\Gm$. Moreover, the graph $\nu$ is plane.

\item 
For $\nu\in\NST(\Gm)$, let the path $\nuf=\sigma(\nu)$ on $\Gd$ be 
obtained from $\nu$ by replacing
any diagonal $\de(v,w)$ in a face $F$ by the pair $\langle v,\phi(F)\rangle,
\langle \phi(F),w\rangle$ of edges. The  function $\sigma$ maps $\NST(\Gm)$ into $\NST(\Gd)$
and is an injection.
The set $\NST(\Gd)$ may be expressed as the disjoint union
\begin{equation}\label{eq:equal}
\NST(\Gd) = \sigma(\NST(\Gm))\cup \NST_2(\Gd),
\end{equation}
where $\NST_2(\Gd)$ is the subset of $\NST(\Gd)$ containing all $\nuf$ for which, for some
face $F$ of $G$, we have (i) $\phi(F)\notin \nuf$, and (ii) 
the intersection $\nuf\cap \pd F$ contains a pair of non-adjacent vertices.

\item Let $\ST(H)$ denote the set of separating triangles of a plane graph $H$.
We have that $\ST(G)\subseteq \ST(\Gd)$, and moreover
\begin{equation}\label{eq:equal2}
\ST(\Gd) = \ST(G)\cup \ST_2(\Gd),
\end{equation}
where $\ST_2(\Gd)$ is the set of all non-facial $3$-cycles of $\Gd$ comprising two
edges of the form $\langle u,\phi(F)\rangle$, $\langle v,\phi(F)\rangle$ for some face $F$ of $G$ and
some $u,v\in \pd F$ with $d_{\pd F}(u,v)\ge 2$, together with an edge $\langle u,v\rangle$ of $G$. 

\item
Let $\nuf$ be a finite \nst\ path of $\Gd$. 
There exists a subsequence of $\nuf$ with the same
endvertices that forms a \nst\ path $\nu$ of $\Gm$.

\end{letlist}
\end{lemma}

We note some further notation. Firstly, the process used in the proof of (d), to replace $\nuf$ by $\nu$, is
termed \emph{$\phi$-removal}. Secondly, since we shall be interested in the mapping $\sigma$,
we introduce another binary relation on the vertex-set $\Vd$ of $\Gd$, namely:
\begin{equation}\label{eq:relation}
\text{for $x,y\in\Vd$, we write $x\hsim y$ if $G$ has some facial cycle $C$ such that $x,y\in\ol C$.}
\end{equation}
The negation of $\hsim$ is written $\nhsim$. 
For $x,y\in V$, we have $ x\hsim y$ if and only if $x$, $y$ are neighbours in $\Gm$.

\begin{proof}
(a) This was proved at \cite[Lemma 4.4]{GL-match}. 
Such $\nu$ cannot contain three or more vertices of any given face since that would contradict the
\nst\ property. If $\nu$ contains two such vertices, it must contain the corresponding edge.
If $\nu$ were non-planar, it would contain two or more diagonals of some face.  

(b) That $\sigma$ is an injection into $\NST(\Gd)$ holds by (a) and the 
obvious invertibility of $\sigma$.
Equation \eqref{eq:equal} holds by a consideration of \dinst s $\nu\in\NST(\GDd)\sm\sigma(\NST(\Gm))$.

(c) The inclusion holds since $G$ is a subgraph of $\Gd$.
Let $T\in\ST(\Gd)\sm\ST(G)$. Since $T \notin \ST(G)$, it contains some edge
of the form $\langle u, \phi(F)\rangle$. Since it is a separating $3$-cycle, it contains a further edge of the form 
$\langle v,\phi(F)\rangle$ where $d_{\pd F}(u,v) \ge 2$. The claim of \eqref{eq:equal2} follows.

(d) Let $\nuf$ be as given, and view it as a directed path. 
If $\nuf\in\sigma(\NST(\Gm))$, we simply replace the facial
sites in $\nuf$ by the corresponding diagonals. Assume that $\nuf\in\NST_2(\Gd)$, and let $F$
be a face of $G$ such that $\phi(F)\notin \nuf$ and $\nuf\cap \pd F$ contains two (or more) non-adjacent vertices.
Let $x$ (\resp, $y$) be the first (\resp, last) vertex of $\nuf$
in $\pd F$, and note that $x\hsim y$. We delete from $\nuf$ the subpath lying between $x$ and $y$ 
while retaining these two vertices and adding the 
corresponding edge (this edge lies in $E$ if $x \sim y$ in $G$,
and is a diagonal otherwise).  This process is iterated for each such face, and
the ensuing path is as required.
\end{proof}

\begin{figure}
\centerline{\includegraphics[width=0.3\textwidth]{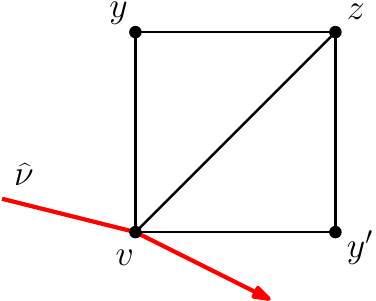}}
\caption{The $4$-cycle in Proposition \ref{lem:main}(b) comprises two triangles with a common edge}
\label{fig:common}
\end{figure}

\section{The main proposition}\label{sec:princ}

We present here the main proposition, which will be used twice in the proof of Theorem \ref{thm:0}. 
The proof of the proposition is deferred to Sections \ref{sec:step2} and \ref{sec:step3}.

\begin{proposition}\label{lem:main}
Let $G\in\sG$  satisfy $\GD\in\sG$.

\begin{letlist}
\item Let $F$ be a face of $\GD$ with four or more edges, and 
let  $\nu\in\NST(\Gm)$ be a path that  includes some vertex $v\in\pd F$. 
There exists $\ol\nu\in\NST(\Gm)$ that includes some diagonal of $F$.

\item
Let $Q$ be a $4$-cycle of $\GDd$ comprising the union of two triangles with 
a common edge $\langle v,z\rangle$ (as in Figure \ref{fig:common}), and 
let  $\nuf\in\sigma(\NST(\Gm))$ be a path that includes no facial site but includes $v$. 
Either there exists $\nuf_1\in\sigma(\NST(\Gm))$ that includes some facial site, or
there exists $\nuf_1\in\sigma(\NST(\Gm))$ that includes 
no facial site but includes 
$z$.
\end{letlist}
Furthermore, the pair $\nu$, $\ol\nu$ (\resp, $\nuf$, $\nuf_1$) differ on only finitely many edges.
\end{proposition}

\section{Cycle structure of a plane graph}\label{sec:step0}

First, we explain how to define the so-called \lq exterior cycle' of a cycle of a plane graph.
This is followed by a description of a system of nested cycles surrounding a given cycle 
of a triangulation.

\subsection{Exterior cycles}\label{ssec:effective}

Let $G=(V,E)\in\sG$, and recall the set $\sC=\sC(G)$ of cycles of $G$.
For $A\in\sC(G)$, we shall construct a new cycle $B=\Ext(A)$ called the \emph{exterior
cycle} of $A$. An intuitive explanation of this is as follows (see Figure \ref{fig:effective}).
As we walk around the cycle $A$,
we may encounter \lq shortcuts' using no edge of $\ins(A)$ --- these are edges not in $A$
that join two vertices of $A$. When allowing the walker to take such shortcuts, the exterior cycle is
the walker's route that traverses fewest edges (including shortcuts). 

Let $A\in\sC(G)$, let $X$ be the set of edges of $G$ of the form $f=\langle a,b\rangle$ with $a,b\in A$;
in particular, the edges of $A$ lie in $X$. 
Let $Y$ be the subset of $X$ containing edges that neither lie in nor intersect
$\ins(A)$.  Recalling that $G$ is embedded in the plane, an edge $f=\langle a,b\rangle\in Y$ may appear either clockwise or anticlockwise around $A$ (in that, when considered as a directed edge from $a$ to $b$,
it has two distinct possible placements in the embedding).
Consider the subgraph of $G$ with edge-set $X$ and its incident vertices, denoted as $X$ also.
Then $X$ has an outer cycle formed of edges in $Y$, and we write $B=\Ext(A)$ for this cycle. 
Note that the number of edges in $B$ is no greater than the number in $A$.
Furthermore, if $A$ is a $3$-cycle, then $A=\Ext(A)$.
The construction is illustrated in Figure \ref{fig:effective} in the case when $A$ is facial in $G$. 

\begin{figure}
\centerline{\includegraphics[width=0.7\textwidth]{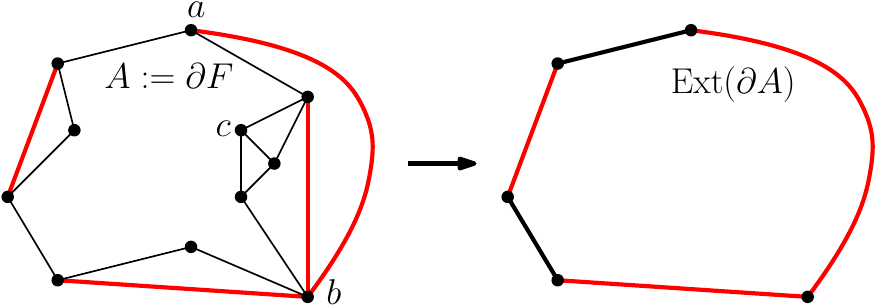}}
\caption{\emph{Left}: A face $F$ of $\GD$ surrounded by the (black) cycle $A:=\pd F$, and 
with further edges \rev{in $Y$} coloured red.
\emph{Right}: The exterior cycle $\Ext(\pd F)$.}\label{fig:effective}
\end{figure}

\begin{remark}\label{rem:facial}
Let $G\in\sG$ satisfy $\GD\in\sG$, and let $A$ be a cycle of $\GD$ (and hence of $G$ also).
We may use either $G$ or $\GD$ in constructing $\Ext(A)$, and the outcome is the same.
If, in addition,  $A$ is a facial cycle of $G$ (that is, $A=\pd F$ for some face $F$ of $G$), 
then $B:=\Ext(A)$ is a cycle of $\GDd$
whose interior contains only one facial site and its incident edges.
We may denote this facial site $\phi(F)$.
\end{remark}

\begin{lemma}\label{lem:201}
Let $G\in\sG$ satisfy $\GD\in\sG$.
Let $F$ be a face of $\GD$ (and hence of $G$ also) with size $4$ or more.
\begin{letlist}
\item
 The exterior cycle $\Ext(\pd F)$
is a cycle of $\GD$ with length $4$ or more.

\item
Let $G(F)$ (\resp, $\GD(F)$) be obtained from $G$ (\resp, $\GD$) by removing 
all vertices and incident edges within $\ins(\Ext(\pd F))$. Then
$\NST(\GD) = \NST(\GD(F))$.

\item
Let $\Gd(F)_\De$ be obtained from the facial graph $\Gd(F)$ of $G(F)$ by emptying
its separating triangles (that is, $\Gd(F)_\De := (\what{G(F)})_\De$).
Then $\Gd(F)_\De=\GDd$. 
In particular, $\NST(\GDd) = \NST(\Gd(F)_\De)$.

\end{letlist}
\end{lemma}

\begin{remark}\label{rem:10}
Let $G\in\sG$ satisfy $\GD\in\sG$.
By Lemma \ref{lem:201}(a), the exterior cycle of a $4$-cycle $Q$ of $\GD$ is $Q$ itself.
\end{remark}

\begin{proof}[Proof of Lemma \ref{lem:201}]
(a) The length $l$ of the cycle $\Ext(\pd F)$ satisfies $l\ge 1$. Evidently, $l\ge 3$ since
$G$ is simple. If $l=3$, then $\Ext(\pd F)$ is a $3$-cycle of $\GD$ whose interior intersects $\pd F$, 
in contradiction of the definition of $\GD$.

(b) Let $\pi \in \NST(\GD)\sm\NST(\GD(F))$, so that $\pi$ contains some vertex, $c$ say, in $\ins(\Ext(\pd F))$.
Thinking of $\pi$ as a directed path, let $a$ be the last vertex of $\pi\cap\Ext(\pd F)$ prior to $c$,
and $b$ the first vertex of $\pi\cap\Ext(\pd F)$ after $c$ (see Figure \ref{fig:effective}). 
By the definition of exterior cycle, and the fact that $\pd F$ is a facial cycle, it is the case that
$a$ and $b$ are neighbours in $\Ext(\pd F)$. Therefore $\pi\notin \NST(\GD)$, a contradiction,
whence $\NST(\GD) \subseteq  \NST(\GD(F))$.
Conversely, any $\pi\in \NST(\GD(F)) \sm \NST(\GD)$ must contain non-consecutive vertices 
$a$, $b$ lying in $\Ext(F)$ such that $\ins(\GD(F))$ contains a path joining $a$ and $b$.
As in the above, this requires that $a$ and $b$ are adjacent in $\Ext(\pd F)$, a contradiction. 

(c) The graphs $G$ and $G(F)$ differ only on the interior of $\Ext(F)$.
Therefore, the same holds for their facial graphs $\Gd$ and $\Gd(F)$.
After emptying separating triangles, each of the two interiors of $F$ in the two resulting graphs
is a wheel with a hub at the facial site $\phi(F)$ (recall Remark \ref{rem:facial})
and spokes to the vertices of $\Ext(F)$. 
It follows that $\Gd(F)_\De = \GDd$ as claimed.
\end{proof}

\subsection{Cycle structure of a triangulation}\label{ssec:step1}

Let $H\in\sG$ be a triangulation.
For $A\in\sC(H)$, we write $N_A$ for the set of neighbours of members of $A$
lying in the unbounded component of $\RR^2\sm A$. Thus, $A \cap N_A=\es$ and (since $G$ is a triangulation)
every $a\in A$ has some neighbour $b\in N_A$.
We think of the edges between $A$ and $N_A$ 
as ordered cyclically as one traverses $A$ clockwise.

The following lemma and more was proved in \cite[Sect.\ 3]{Jung}, from which we extract the element of 
current interest.

\begin{lemma}\label{lem:101}

Let $H\in\sG$ be a triangulation, and let $A\in\sC(H)$. The set $N_A$ contains a cycle $B\in\sC(H)$ 
satisfying $A\subseteq \ins(B)$ and $N_A\subseteq \ol B$.
\end{lemma}

\begin{proof}
Consider the finite graph $J$ induced by the vertices of $H$ in $\ol A \cup N_A$.
By construction, $J$ is connected, and is
an inner triangulation (in that it is finite and all its faces except possibly its exterior face are triangles).
The exterior
face is the unique unbounded face, and it has a boundary $B$ comprising edges of $J\sm \ol A$. 
The set $B$ forms a cycle since, if not, 
$J$ contains some $c$ such that $c\notin \ol A$ and $d_H(A,c)\ge 2$.
This would be a contradiction.
\end{proof}

There follow two lemmas that  will be used in the proof of Theorem \ref{thm:0}
at the end of this section. Recall from Lemma \ref{lem:4} the map $\sigma:\NST(\Gm) \to \NST(\Gd)$.

\begin{lemma}\label{lem:102}

Let $G\in\sG$ satisfy $\GD\in\sG$, and let $A$ be a cycle 
of the triangulation $\GDd$. Assume $\Gm$ has a \dinst\ $\nu$
such that $\nuf=\sigma(\nu)$ has the following properties: (i) $\nuf$ includes
no facial site, (ii) $\nuf\cap N_A\ne\es$, and (iii) $\nuf\cap A=\es$. 
There exists $\ol\nu\in\NST(\Gm)$ such that either (i) $\ol\nu$ traverses some diagonal, or (ii)
$\ol\nu$ traverses no diagonal but satisfies $\nuf\cap A\ne\es$.
Furthermore, $\nu$ and $\ol\nu$ differ on only finitely many edges.
\end{lemma}

\begin{proof}[Proof of Lemma \ref{lem:102} using Proposition \ref{lem:main}(b)]
Let $\nu\in\NST(\Gm)$ be as given. 
Since $\nuf\cap N_A\ne\es$ by assumption, we have that $\nuf\cap B\ne\es$ also
(where $B$ is given in Lemma \ref{lem:101} with $H=\GDd$). Let $v$ ($\in V$)
be the first  point  in $\nuf$ (considered as a directed path)
that lies in $B$.

Since $B\subseteq N_A$, there exists an edge $e=\langle v,z\rangle$ of $\GDd$ 
with $z\in A$.
The edge $e$ lies in two $3$-cycles of $\GDd$,
and the union of these triangles forms a quadrilateral $Q$
with $v$ and $z$ as opposite vertices. See Figure \ref{fig:twosits}.
The claim follows by Proposition \ref{lem:main}(b).
\end{proof}

\begin{figure}
\centerline{\includegraphics[width=0.3\textwidth]{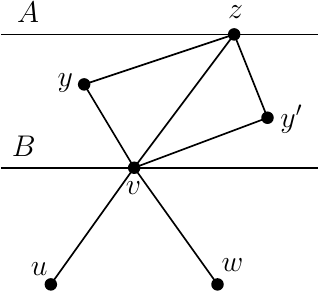}}
\caption{When $v\in \nuf\cap B$, there exists $z\in A$ such
that $v\sim z$ in $\GDd$.  The edge $\langle v,z\rangle$ lies in two triangles whose union 
forms the quadrilateral illustrated here. Each of the vertices $y$, $y'$ may lie in either $A$
or $B$ or neither.}
\label{fig:twosits}
\end{figure}

\begin{lemma}\label{lem:103}
Let $G\in\sG$ satisfy $\GD\in\sG$, and let $A$ be a cycle of $\GD$ (and hence of $G$ also) of size $4$ or more.
If $\Gm$ has some \dinst\ $\nu$, then either 
(i) there exists $\ol\nu\in\NST(\Gm)$ that traverses some diagonal, or 
(ii) there exists $\ol\nu\in\NST(\Gm)$ that traverses no diagonal but includes some vertex of $A$.
Furthermore, $\nu$ and $\ol\nu$ differ on only finitely many edges.
\end{lemma}

\begin{proof}[Proof of Lemma \ref{lem:103} using Proposition \ref{lem:main}(b)]

Let $A'=\Ext(A)$ be the exterior cycle of $A$.
By iteration of Lemma \ref{lem:101} applied to the triangulation $\GDd$, 
there exists a sequence $A_0, A_1,A_2,\dots$ of cycles in 
$\GDd$ such that $A_0=A'$ and, for $i\ge 0$,
$A_i\subseteq \ins(A_{i+1})$ and $A_{i+1}\subseteq N_{A_i} \subseteq \ol{A_{i+1}}$.   
Since $G\in\sG$ and $A_{i}\subseteq \ins(A_{i+1})$, 
\begin{equation}
V\cap \ins(A_i) \uparrow V \qq\text{as } i\to\oo.
\label{eq:101}
\end{equation}

Let $\nu\in\NST(\Gm)$ and $\nuf=\sigma(\nu)$. If $\nu$ traverses some diagonal, we may take $\ol\nu=\nu$.
Assume that $\nu$ traverses no diagonal, so that $\nuf$ includes no facial site.

By \eqref{eq:101}, there exists $I$ 
such that $\nuf\cap A_I\ne\es$,
and we pick $I=I(\nuf)$ minimal with this property. 
If $I=0$, there is nothing to prove since $A_0=A'$ and $A'\subseteq A$.
Assume then that $I\ge 1$, and let $v\in \nuf\cap A_I$; since $\nuf$ includes no facial site, we have $v\in V$. 
By Lemma \ref{lem:102}, there exists $\nu'\in\NST(\Gm)$ such that either 
(i) $\nu'$ traverses some diagonal, or (ii)
$\nu'$ traverses no diagonal but $\nuf':=\sigma(\nu')$ satisfies $\nuf'\cap A_{I-1}\ne\es$.
If (i) holds, the proof is complete. Otherwise,
$\nuf'$ satisfies 
$I(\nuf') \le I-1$.

We continue by iteration. At each stage we could possibly 
obtain some $\ol\nu\in\NST(\Gm)$ that traverses some diagonal.
If this occurs at no stage of the iteration,  we obtain finally some
$\nu''\in\NST(\Gm)$ that traverses no diagonal, and such that $\nuf''=\sigma(\nu'')$ satisfies $I(\nuf'')=0$ and 
$\nuf'' \cap A' \ne\es$. The claim follows since $A'\subseteq A$.
\end{proof}

\begin{proof}[Proof of Theorem \ref{thm:0} using Proposition \ref{lem:main}]
Let $G\in\sG$ be such that $\GD\in\sG$ is not a triangulation, and let $F$ be a
face of $\GD$ of size $4$ or more. Let $\nu\in\NST(\Gm)$. If $\nu$ traverses some diagonal
 then the proof is complete, so we may assume that $\nu$ traverses no diagonal.
By Lemma \ref{lem:103}, there exists $\nu_1\in\NST(\Gm)$ that traverses no diagonal
but intersects $\pd F$. 
We apply Proposition \ref{lem:main}(a) to complete the proof.
\end{proof}

\section{Proof of Proposition \ref{lem:main}($\mathrm a$)}\label{sec:step2}

We begin with an outline. 
Let $G$ be as in the statement. 
Since $\GD$ is not a triangulation, it has some face $F$ of size $4$ or more (note that $F$ is
also a face of $G$).
Let $\nu$ and $v$ be as in the statement of part (a).
We shall explain how to make local changes to $\nu$ to
obtain a \dinst\ $\ol\nu$ of $\Gm$ that agrees with $\nu$ 
except on finitely many edges, and that contains some diagonal of $\pd F$. 
This will be done in the universe of \nst\ paths on the facial graph $\GDd=(\VDd,\EDd)$.
Let $\nuf=\sigma(\nu)$ be the \dinst\ of $\GDd$
corresponding to $\nu$ (see Lemma \ref{lem:4}(b) and \eqref{eq:noint2}).
We shall make local changes to $\nuf$ to obtain
a \dinst\ $\nuf_1 \in \sigma(\NST(\Gm))$ that includes the facial site $\phi(F)$. 
The path $\ol\nu=\sigma^{-1}(\nuf_1)\in\NST(\Gm)$ has the required property. 
There are a number of steps in the pursuit of this strategy, as follows.

Let $G=(V,E)$, $v$, $F$ be as above, and
let $\nu=(\dots, \nu_{-1}, \nu_0,\nu_1,\dots)\in\NST(\Gm)$ with $\nu_0 =v$;
it is sometimes convenient to think of $\nu$ as a \emph{directed} path. 
We may assume that
\begin{equation}\label{eq:nst*}
\text{$\nu$ contains no diagonal of $F$,}
\end{equation}
since otherwise there is nothing to prove.
Therefore, by Lemma \ref{lem:4}(a),
\begin{equation}\label{eq:nst**}
\text{$\nu\cap \pd F$ comprises either a single vertex of $V$ or a single edge of $E$.}
\end{equation}

Rather than working with the boundary cycle $\pd F$ of the face $F$, we shall work with its exterior cycle 
$E=\Ext(\pd F)$. The latter cycle is facial (in both $G(F)$ and $G(F)_\De$,
recall  Lemma \ref{lem:201}(b)) and we denote this face by $F'=\ins(E)$,
so that $E=\pd F'$.  
Recall from Lemma \ref{lem:201}(c) that $\Gd(F)_\De = \GDd$.
By Remark \ref{rem:facial}, we may take $\phi(F')=\phi(F)$.
Let $\nuf=\sigma(\nu)\in\NST(\GDd)$.
By \eqref{eq:nst*} and 
\eqref{eq:nst**},
\begin{gather}
\text{$\nuf$ does not contain the facial site $\phi(F')$},\label{eq:nst*f}
\\
\text{$\nuf\cap \pd F'$ comprises either a single vertex of $V$ or a single edge of $E$}.
\label{eq:nst**f}
\end{gather}

In the various steps
and figures that follow, we write
\begin{equation*}
u=\nu_{-1},\q v=\nu_0,\q w= \nu_{1}.
\end{equation*}
Represent the triple $u$, $v$, $w$ in the plane graph $\GDd$ as in Figure \ref{fig:basic},
so that $F'$ lies \lq above' the triple ($F'$ is depicted in the figure with
its facial site and incident edges removed). Let $f_i=\langle v, y_i\rangle$,
$i=1,2,\dots, r$, be the edges of $\GDd$ incident to $v$
in the sector obtained by rotating $\langle u,v\rangle$ clockwise about $v$ until it coincides with $\langle w,v\rangle$;
the $f_i$ are listed in clockwise order.  
Since $\GDd$ is simple, the $y_i$ are distinct.

For a (directed) path $\pi$ and a vertex $x\in\pi$, let $\pi(x-)$ (\resp, $\pi(x+)$) be the subpath of $\pi$
prior to and including $x$ (\resp, after and including $x$).

There are two cases to consider depending on which case of \eqref{eq:nst**f} holds 
(see Sections \ref{ssec:neither} and \ref{ssec:either}). 
If $\nuf\cap\pd F'$ is a singleton $v$ (as in Figure \ref{fig:basic}),
we denote by $y_N$ and $y_{N+1}$ the two
neighbours of $v$ lying in $\pd F'$ (in particular, we have $y_N\ne y_{N+1}$).
If $\nuf\cap \pd F'$ is an edge, we may take that edge to be $\langle v,w\rangle$, and we denote by $y_N$
the vertex of $\pd F'$ other than $w$ that is incident to $v$ (as in Figure \ref{fig:8});
in this case we have $N=r$. 

\begin{lemma}\label{lem:202}\mbox{\hfil}
\begin{letlist}
\item
Let $s_0=u$, $s_{r+1}=w$, and $s_i=y_i$ for $i=1,2,\dots,r$.
If $s_i \sim s_j$ then $\vert i-j\vert=1$. Conversely, $s_0\sim s_1\sim\cdots\sim s_N$ and
$s_{N+1}\sim\cdots\sim s_{r+1}$, where $N$ is such that $y_N$ and $y_{N+1}$ are the two neighbours
of $v$ lying in $\pd F'$.
\item
If $y_i\in \pd F'$, then $i\in\{N,N+1\}$.

\item No $y_i$ lies in $\nuf(u-)\cup \nuf(w+)$. 
\end{letlist}
\end{lemma}

\begin{proof}
(a) Suppose $s_i\sim s_j$ where $j\ge i+2$. Then $(v,s_i,s_j)$ forms a $3$-cycle $T$ of 
the triangulation $\GDd$ whose interior intersects the edge $\langle v,s_{i+1}\rangle$. 
This is a contradiction since $\GDd$
is $\tri$-empty. The partial converse holds as stated since $\GDd$ is a triangulation.

(b) This holds by the definition of the exterior cycle $F'=\Ext(F)$.

(c) This follows from the fact that $\nu$ is \nst\ in $\Gm$.
\end{proof}

  \begin{figure}
 \centerline{\includegraphics[width=0.4\textwidth]{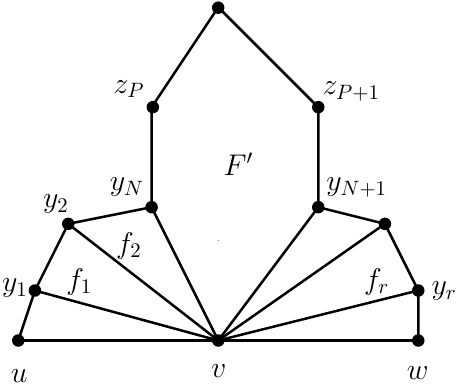}}
\caption{The path $\nuf$ passes through a vertex $v$ that lies in the boundary of a $6$-face $F'$.
This is an illustration of the case when neither $\langle u,v\rangle$ nor $\langle v,w\rangle$
lie in $\pd F'$.}
\label{fig:basic}
\end{figure}

  \begin{figure}
 \centerline{\includegraphics[width=0.6\textwidth]{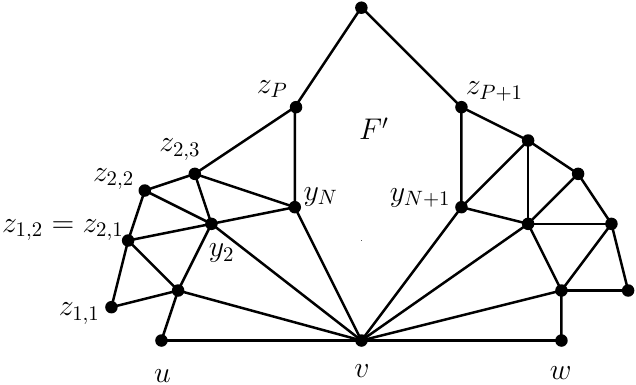}}
\caption{An illustration of the vertices $z_{i,j}$. Note that $z_{i,\de_i}=z_{i+1,1}$ for $i\ne N$, as
in \eqref{eq:edgessame}, and that any two consecutive $z_{i,j}$ 
(other than $(z_P,z_{P+1})$) close a triangle, as in \eqref{eq:edgein}.
This illustration is a simplification --- see Figure \ref{fig:basic3}.}
\label{fig:basic2}
\end{figure}

 \begin{figure}
 \centerline{\includegraphics[width=0.6\textwidth]{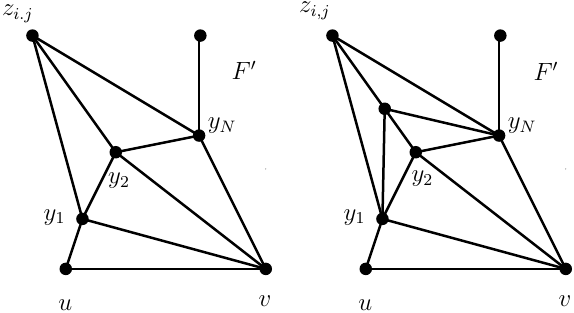}}
\caption{On the left, there is a vertex $z_{i,j}$ connected to each of $y_1,y_2,\dots,y_N$.
On the right, the relationship of this vertex  to $y_2$ is more complicated.}
\label{fig:basic3}
\end{figure}

For $i=1,2,\dots,r$, denote the neighbours of $y_i$ other than possibly $u,v,y_{i-1},y_{i+1},w$ as 
$z_{i,1},z_{i,2},\dots, z_{i, \de_i}$, listed in clockwise order of the planar embedding.
More precisely, we list the edges exiting $y_i$ (other than any edges to $u, v, y_{i-1}, y_{i+1},w$) 
according to clockwise order, and we denote the other endvertex of the $j$th such edge as $z_{i,j}$.
Note that, while the $z_{i,1},z_{i,2},\dots, z_{i, \de_i}$ are distinct for given $i$
(since $\GDd$ is simple), there may generally 
exist values of $i \ne j$ 
and $1\le a\le \de_i$, $1\le b\le \de_j$  with $z_{i,a}=z_{j,b}$. 

We list the labels $z_{i,j}$ in lexicographic order (that is, $z_{a,b}<z_{c,d}$ if either $a<c$, or $a=c$ and $b<d$)
as $z_1<z_2<\dots<z_s$; this is a total order of the \emph{label-set}
but not necessarily of the underlying \emph{vertex-set} since a given vertex may occur multiple times
(if two labelled vertices $z_a$, $z_b$ satisfy $z_a=z_b$ when viewed as vertices, 
we say that each label is an \emph{image} of the other). 
If $a<b$ we speak of $z_a$ as preceding, or being to the \emph{left} of $z_b$ (and $z_b$ succeeding, or being to the
\emph{right} of $z_a$). For $1\le i\le r$, let 
\begin{equation}\label{eq:defS}
\text{$S_i = (z_{i,j}: j=1,2,\dots,\de_i)$, viewed as an ordered subsequence of $Z$.}
\end{equation}
Since $\GDd=(\VDd,\EDd)$ is a triangulation,  
\begin{equation}\label{eq:edgein}
\langle z_{i,j}, z_{i,j+1}\rangle \in \EDd, \qq j=1,2,\dots, \de_i-1,\ 1\le i\le r,
\end{equation}
and moreover
\begin{equation}\label{eq:edgessame}
z_{i,\de_i} = z_{i+1,1} \qq 1\le i <r,\ i\ne N,
\end{equation}
whenever the relevant pair of vertices is defined. 

As in Figure \ref{fig:basic2}, let $y_N, y_{N+1}$ be the two neighbours of $v$  
in $\pd F'$, and $z_P$, $z_{P+1}$ 
their further neighbours in $\pd F'$ 
(if $F'$ is a quadrilateral, we have $z_P=z_{P+1}$). 
It can be the case that $z_i\in\pd F'$ for some $i\notin\{P,P+1\}$.

See Figures \ref{fig:basic2} and \ref{fig:basic3} for
illustrations of the $z_{i,j}$.
By \eqref{eq:edgein}--\eqref{eq:edgessame}, 
\begin{equation}\label{eq:edgesin2}
\text{$\pi_u=(u,z_1,z_2,\dots,z_P)$ and $\pi_w=(z_{P+1},\dots,z_s,w)$ are walks of $\GDd$.}
\end{equation}
Note that $\pi_u$ and $\pi_w$ may contain cycles and oxbows, and may intersect one another.
Further information concerning the relationship between the $z_i$ and the $y_j$ may be
gleaned from \cite[Sect.\ 3]{Jung}.

In making changes to the path $\nuf$, it is useful to first
record which vertices lie in either $\nuf(u-)$ or $\nuf(w+)$, or in neither.
We label each vertex $x\in V$ by 
\begin{equation*}
\begin{cases} U &\text{if $x\in\nuf(u-)$},\\
W &\text{if  $x\in \nuf(w+)$},\\
Q &\text{if $x\notin \nuf(u-) \cup \nuf(w+)$}.
\end{cases}
\end{equation*}
Write $N_L$ be the number of $z_i$ with label $L$.
Since $\nu\in\NST(\Gm)$,  by \eqref{eq:nst*}
\begin{equation}\label{eq:notin1}
\text{every $x\in \pd F'$ satisfying $x\ne v,w$ is labelled $Q$.}
\end{equation}

\begin{figure}
\centerline{\includegraphics[width=0.6\textwidth]{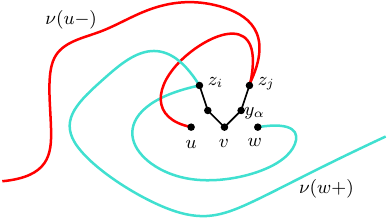}}
\caption{If $z_i\in\nuf(w+)$ and $z_j\in\nuf(u-)$ where $i<j$, 
then the pair $\nuf(u-)$, $\nuf(w+)$ fails to be \nt, and is indeed intersecting.}\label{fig:nonp}
\end{figure}

According to \eqref{eq:nst**} there are two cases, which we consider in order.

\subsection{Case I: Suppose $\pd F'$ contains neither of the edges $\langle u,v\rangle$, $\langle v,w\rangle$.}
\label{ssec:neither}

This case is illustrated in Figure  \ref{fig:basic}. 
Here is a technical lemma.

\begin{lemma}\label{lem:notri}\mbox{\hfil}
Suppose $N_U\ge 1$ and let $z_{\rho}=z_{\alpha,\beta}$ be the rightmost $z_i$ with label  $U$. 
Let $\nuf_{\rho}''(u-)$ be the subpath of $\nuf(u-)$ from $z_{\rho}$ to $u$,
and $\nuf_{\rho}'(u-)$ that obtained from $\nuf(u-)$ by deleting $\nuf_{\rho}''(u-)$
while retaining its endpoint $z_\rho$. 

\begin{letlist}
\item The path $\nuf_\rho''(u-)$ 
moves around $v$ in an anticlockwise direction in the sense that the directed cycle $D$ obtained
by traversing $\nuf_\rho''(u-)$  from
$z_\rho$ to $u$, followed by the edges $\langle u,v\rangle$, $\langle v,y_\alpha\rangle,
\langle y_\alpha,z_\rho\rangle$, has 
winding number $-1$. Furthermore, $\ol D \cap \nuf(w+)=\es$.

\item For $1\le i \le \rho-1$, $z_i$ is labelled either $Q$ or $U$.

\item For $1\le  i\le \rho-1$, $z_i$ has no $\GDd$-neighbour lying in $\nuf_\rho'(u-)$,
apart possibly from $z_\rho$ or one of its images. 
Moreover, for all $x\in\nuf'_\rho(u-)\sm\{z_\rho\}$,
we have that $z_i \nhsim x$.

\item 
For $1\le  i\le \rho$, $z_i$ has no $\GDd$-neighbour lying in $\nuf(w+)$.
Moreover, for all $x\in\nuf(w+)$,
we have that $z_i \nhsim x$.

\end{letlist}
\end{lemma}

\begin{figure}
\centerline{\includegraphics[width=0.6\textwidth]{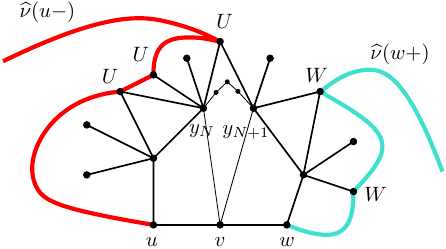}}
\caption{The path $\nuf(u-)$ intersects the $z_a$ at the rightmost $z_\rho$ and then progresses 
anticlockwise to $u$. 
Similarly $\nuf(w+)$
hits at the leftmost vertex $z_\lam$ and progresses clockwise to $w$.}\label{fig:explan6}
\end{figure}

\begin{proof}

(a)  If the given cycle has winding number $1$, then
$\nuf_\rho''(u-)$ intersects
$\nuf(w+)$ in contradiction of the definition of $\nuf$.
See Figure \ref{fig:nonp}. The final claim holds since $v,y_N \notin \nuf(w+)$.

(b) Let $1\le i\le \rho-1$.  If $z_i \in \nuf(w+)$,
then (as illustrated in Figure \ref{fig:nonp}), $\nuf(u-)$ and $\nuf(w+)$ must intersect 
(when viewed as arcs in $\RR^2$). This contradicts the planarity of $\nuf$. Therefore,
such $z_i$ is labelled either $Q$ or $U$.

(c) Let $1\le i\le \rho-1$ and  suppose $z_i$ has a $\GDd$-neighbour $x$ (with $x$ a different vertex from $z_\rho$) 
belonging to  $\nuf_\rho'(u-)$.
By a consideration of the cycle $D$ of Lemma \ref{lem:notri}(a), we have that  $d_{\GDd}(x,\nuf_\rho''(u-))
\le 1$, which (as above) contradicts
the fact that $\nuf(u-)$ is \nst\ in $\GDd$. The second statement holds similarly, since
$\nu\in\NST(\Gm)$.

(d) This is similar to (c) above. 
\end{proof}

\begin{figure}
\centerline{\includegraphics[width=0.55\textwidth]{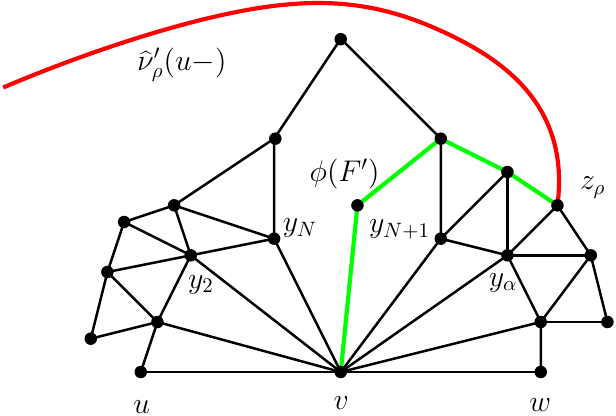}}
\caption{An illustration of $\nuf_1$ in case (A), when the rightmost
$z_i$ labelled $U$ is to the right of $z_{P+1}$.}
\label{fig:basic2n}
\end{figure}

Similar conclusions hold with $U$ replaced by $W$, 
and $z_\rho$ replaced by the leftmost $z_\lam$ in $\nuf(w+)$. 
 See Figure \ref{fig:explan6} for an illustration of Lemma \ref{lem:notri}(b)
 and some of its consequences. 
In its approach towards $u$ (from infinity) $\nuf(u-)$ passes through the rightmost $z_{\rho}$. 
It may subsequently visit one or more $z_i$ with $i<\rho$,
 but it must do this in decreasing order of suffix. Similarly, $\nuf(w+)$ passes through the 
 leftmost $z_{\lam}$ and may
 subsequently visit one or more $z_i$ with $i>\lam$ in clockwise order of suffix. 
 That $\lam>\rho$ (when these suffices are defined) holds by Lemma \ref{lem:notri}(b).
 
Let $z_\rho=z_{a,b}$ be the rightmost $z_i$ labelled $U$ (with $\rho=0$ 
and $z_0 := u$ if $N_U=0$).
Similarly, let $z_\lam=z_{c,d}$ be the leftmost $z_i$ labelled $W$ (with $\lam=r+1$ 
and $z_{r+1}:=w$ if $N_W=0$).
By the \nst\ property of $\nu$ (see also \eqref{eq:notin1}), we have
\begin{equation}\label{eq:not inF}
z_\rho \nhsim z_\lambda,\  z_\rho, z_\lambda \notin \pd  F'.
\end{equation}
For the special case when $\rho=0$ and $\lam=r+1$, we use here the fact that $v$ is not a facial site.

\smallskip\noindent
(A). \emph{Suppose  $\rho \ge P+1$.} 
Let $\alpha=\max\{i\ge N+1: z_\rho\in S_i\}$, say $z_\rho=z_{\alpha,\beta}$
(recall the set $S_i$ from \eqref{eq:defS}).
We add to $\nuf_\rho'(u-)$ the set of vertices 
\begin{equation*}
W:=\{z_{a,b}:  N+1\le a \le \alpha-1,\ 1\le b\le \de_a\} \cup
\{z_{\alpha,j}: 1\le j\le \beta-1\},
\end{equation*}
viewed as an ordered sequence of vertices from $z_\rho$ to $z_{P+1}$. It can be that some $z\in W$
with $z\ne z_{P+1}$ satisfies $z\in \pd F'$. If that holds, we find  such $z$ with greatest suffix and 
remove all elements of $W$ with lesser suffix than $z$. 
Note, in this case, that $z\notin\{y_N,v,y_{N+1},z_{P+1}\}$.
See Figure \ref{fig:basic2n}.

This yields a doubly infinite path 
$\nuf_1 = (\nuf'_\rho(u-), W', \phi(F'),v, \nuf(w-))$ of $\GDd$ where
$W'$ is obtained from $W$ by $\phi$-removal and oxbow removal. We claim that 
$\nuf_1\in\sigma(\NST(\Gm))$. To check this, it suffices to verify that there exist no $x\in(\nuf_\rho'(u-), W')$ and
$y\in\nuf(w+)$ such that $x \hsim y$. 
This follows from Lemma \ref{lem:notri}(d) 
and a consideration based on whether or not $y_\alpha$ is a facial site.

Since $\nuf_1$ includes the facial
site $\phi(F')$,  there exists
$\ol\nu=\sigma^{-1}(\nuf_1)\in\NST(\Gm)$ that traverses a diagonal of $F'$,
as required.

\smallskip\noindent
(B). \emph{Suppose $\lam \le P$.}
This is similar to Case (A).

\smallskip\noindent
(C). \emph{Suppose either $\rho= P$ or $\lam = P+1$.}
Assume $\rho=P$; the other case is similar. By \eqref{eq:not inF},
we may add $\phi(F')$ to $\nuf'_\rho(u-)
\cup\{v\}\cup \nuf(w+)$ to obtain the required \dinst\ $\nuf_1$, and hence $\ol\nu
=\sigma^{-1}(\nuf_1)$ as before.

\begin{figure}
\centerline{\includegraphics[width=0.6\textwidth]{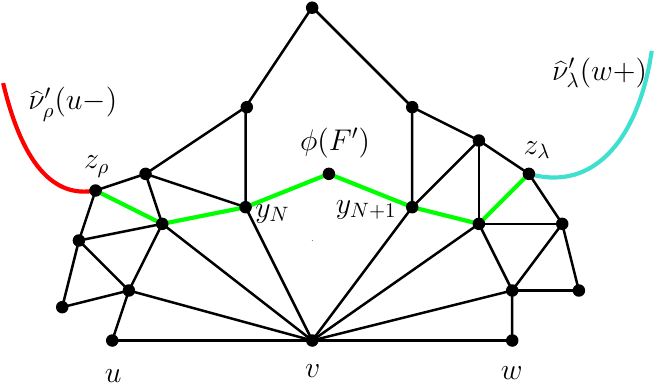}}
\caption{An illustration of $\nuf_1$ in case (D.1),
when the rightmost $U$ lies to the left and the leftmost $W$ lies to the right.}\label{fig:6}
\end{figure}

\smallskip\noindent
(D). \emph{Suppose $\rho < P$ and $\lam > P+1$.}
Write $z_\rho=z_{\alpha,\beta}$ and $z_\lam=z_{\gamma,\delta}$ (with $\alpha=1$ if $\rho=0$,
and $\gamma=r$ if $\lam=r+1$). 
There are two cases, depending on whether or not
\begin{equation}\label{eq:601}
\exists\text{ $i$, $j$ with $\rho<i<P<P+1<j<\lambda$ such that $z_i=z_j=\phi(J)$ for some $J$.}
\end{equation}
\begin{numlist}
\item Assume \eqref{eq:601} does not hold. There is no pair $y_k$, $y_l$ with
$\alpha<k\le N$, $N+1\le l<\gamma$ that lie in the same facial cycle of $\GDd$. 
In this case we remove $\nuf_\rho''(u-)$ and $\nuf''_\lam(w+)$ and add the vertices
$y_\alpha,y_{\alpha+1},\dots,y_N$, $\phi(F')$, and $y_{N+1}, y_{N+2},\dots,y_\gamma$. 
The resulting set of vertices contains 
(after $\phi$-removal and oxbow removal) a 
\dinst\ $\nuf_1\in\sigma(\NST(\Gm))$ that includes the facial site $\phi(F')$.
The required \dinst\ of $\Gm$ is $\ol\nu:=\sigma^{-1}(\nuf_1)$.
See Figure \ref{fig:6}.

\item
Assume that \eqref{eq:601} holds and pick $i$ least and then $j$ greatest.
Write $z$ for the common vertex $z_i=z_j$ where
$z=\phi(J)$ for some face $J$ of $\GDd$. It cannot be that both
$\nuf(u-)\cap\pd J\ne\es$ and  $\nuf(w+)\cap\pd J\ne\es$, since that contradicts
$\nu\in\NST(\Gm)$; assume then that  $\nuf(w+)\cap \pd J=\es$.
See Figure \ref{fig:8}.
\begin{romlist}

\item Suppose there exists $x\in\nuf(u-)\cap \pd J$. By the planarity of $\nuf$,
it must be that $x\in\nuf'_\rho(u-)$, and we pick such $x$ earliest with this property.  We consider the walk 
\begin{equation*}
(\nuf(x-),z_i,z_{i+1}\dots,z_P,
\phi(F'), v,\nuf(w+)).
\end{equation*}
 After $\phi$-removal and oxbow removal, this becomes a \dinst\ $\nuf_1$ of 
$\GDd$ lying in $\sigma(\NST(\Gm))$. The required \dinst\ of $\Gm$ is $\ol\nu:=\sigma^{-1}(\nuf_1)$.

\begin{figure}
\centerline{\includegraphics[width=0.6\textwidth]{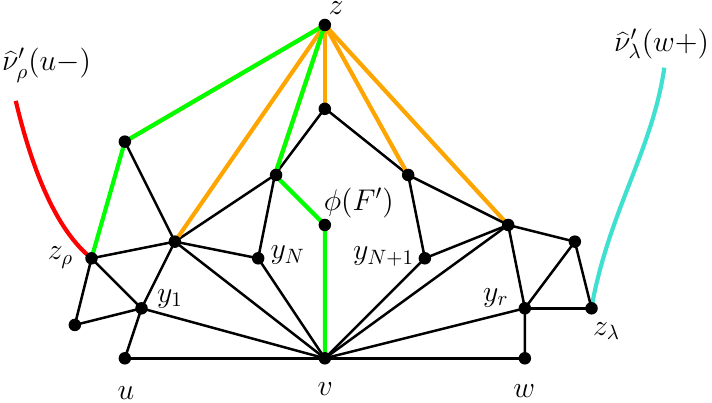}}
\caption{An illustration of $\nuf_1$ in case (D.2)(ii).  The vertex $z$ is a facial site in the face $J$, 
and is joined
to $\pd J$ by the orange edges. The additional path from $z_\rho$ to $v$ is marked in green,
and it makes use of the facial site $\phi(F')$.}\label{fig:8}
\end{figure}

\item

Suppose that $\nuf(u-)\cap \pd J=\es$.
We apply the argument of the above case to the walk
$(\nuf'_\rho(u-),z_{\rho+1},z_{\rho+2},\dots,z_P,\phi(F'), v,\nuf(w+))$.
\end{romlist}
\end{numlist}

\begin{figure}
\centerline{\includegraphics[width=0.8\textwidth]{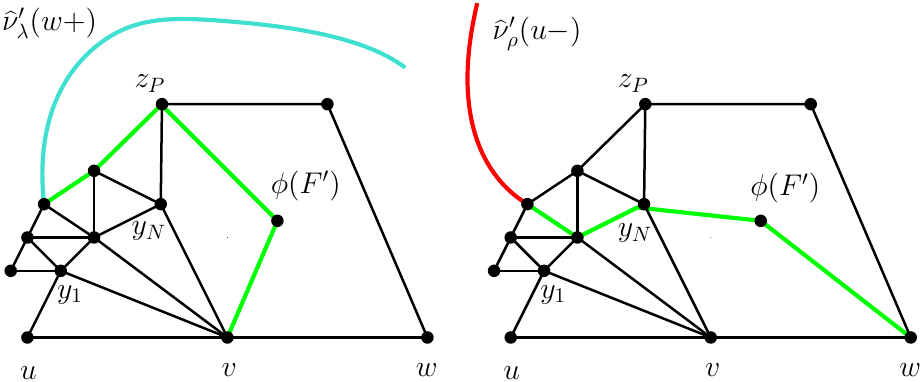}}
\caption{Illustrations of the constructions in Section \ref{ssec:either}.
\emph{Left}: When $\lam\le P$, the path $\nuf'_\lam(w+)$ followed by certain vertices 
as marked results in a 
\dinst\ including the facial site $\phi(F')$. \emph{Right}: 
When $\lam\ge P+1$, the path $\nuf'_\rho(u-)$ followed by certain vertices as marked forms a 
\dinst\ including $\phi(F')$.}\label{fig:9}
\end{figure}

\subsection{Case II: Suppose $\pd F'$ contains $\langle v,w\rangle$ but not $\langle u,v\rangle$.}
\label{ssec:either}

The argument is similar to that of Section \ref{ssec:neither}, 
and we sketch it. Let $y_1,y_2,\dots,y_N$ be
the vertices adjacent to $v$ above the triple $u,v,w$, as illustrated in Figure
\ref{fig:9}.
Let the $z_{i,j}$ be as in the last section, and let $(z_i:1\le i\le  P)$, $z_\rho$, and $z_\lam$ be given as before.

\smallskip\noindent
(E). \emph{Suppose  some $z_{i,j}$ is labelled $W$.} We proceed as in (A), (B) above.
Find the leftmost such vertex, say $z_\lam$. We delete 
$\nuf''_\lam(w+)$ from $\nuf$ and add the $z_{i,j}$ that lie between $z_\lam$ and $z_{P}$. 
This results (after $\phi$-removal and oxbow removal) in  
a \dinst\ $\nuf_1\in\sigma(\NST(\Gm))$ that includes the ordered sequence $(z_P, \phi(F'), \nuf(u-))$. 

\smallskip\noindent
(F). \emph{Suppose no $z_{i,j}$ is labelled $W$.} We proceed as in (D) above.
Find the rightmost $z_{i,j}$  labelled $U$, say $z_\rho=z_{\alpha,\beta}$
 (with $\rho=0$ if no such vertex exists). We delete $\nuf''_\rho(u-)$ and $v$ from $\nuf$
and add $y_\alpha,y_{\alpha+1}\dots,y_N$
to obtain a  \dinst\ $\nuf_1\in\NST(\Gm))$ that includes the ordered triple $(y_N,\phi(F'), w)$.

\begin{figure}
\centerline{\includegraphics[width=0.4\textwidth]{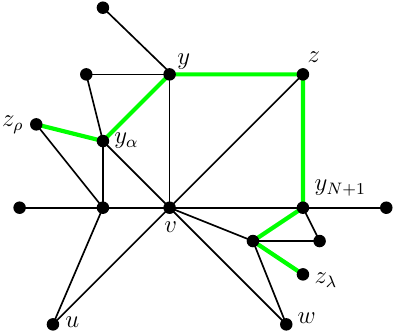}
\quad\includegraphics[width=0.4\textwidth]{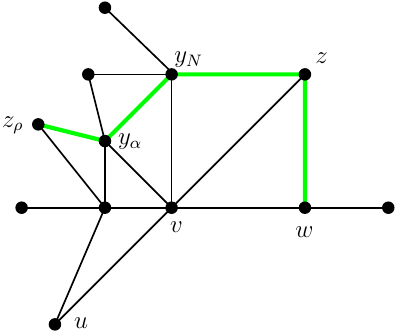}}
\caption{Illustrations of the constructions in Sections \ref{ssec:8.1} and \ref{ssec:8.2}, \resp.}
\label{fig:basic''}
\end{figure}

\section{Proof of Proposition \ref{lem:main}($\mathrm b$)}\label{sec:step3}

Let $Q$ be a $4$-cycle in $\GDd$ as in Figure \ref{fig:twosits}, 
and note that some vertices of $Q$ may be facial sites.
Let $\nuf\in\sigma(\NST(\Gm))$ be such that $v\in \nuf\cap Q$.
If $\nuf$ includes some facial site, there is nothing more to prove, and so we may assume henceforth that
\begin{equation}\label{eq:nofacial}
\text{$\nuf$ includes no facial site.}
\end{equation}
In particular, $u,v,w\in V$.

We may assume that $z\notin \nuf$, since otherwise there is nothing to prove.
In place of \eqref{eq:nst*} we have (in the notation of Figure \ref{fig:twosits})
that $\nuf\cap Q$ is one of
(i) the single vertex $v$, (ii) the single edge $\langle v,y'\rangle$, (iii) the single edge
$\langle v,y\rangle$, (iv) the two edges $\langle v,y'\rangle$, $\langle v,y\rangle$.
Case (iii) is handled as case (ii), and we proceed with cases (i), (ii), (iv) next.

\subsection{(i) Assume that $\nuf\cap Q=\{v\}$,}\label{ssec:8.1}
 and temporarily remove the edge $\langle v,z\rangle$ to obtain a $4$-face $F$ with $\pd F=Q$. 
 We shall reinstate this diagonal later.

We follow the constructions in the proof of Section \ref{ssec:neither}.
With the exception of case (D) of that section, we may take $\nuf_1$ as given there
(with the facial site $\phi(F')$ removed, so that the new path traverses the diagonal of $F$).
Either $\nuf_1$ includes some facial site or it does not, and in either case the claim follows.

We next consider case (D) with the diagonal reinstated in $F$, and see Figure \ref{fig:basic''}.
Vertices $z_\rho=z_{\alpha,\beta}$ and $z_\lambda=z_{\gamma,\delta}$ are as before.
Since $\nu$ is \nst\ and traverses no diagonal,
\begin{equation}\label{eq:610}
\nuf'_\rho(u-)\nhsim \nuf'_\lam(w+).
\end{equation}

\begin{numlist}
\item
If $z \nhsim\nuf'_\rho(u-)\cup \nuf'_\lam(w+)$,
we consider the walk 
\begin{equation*}
w=(\nuf'_\rho(u-),y_\alpha,y_{\alpha+1},\dots,y_N (=y), z,y_{N+1}(=y'), y_{N+2},\dots,y_\gamma,\nuf_\lam'(w+)).
\end{equation*}
It may that $y_i\hsim y_j$ for some $\alpha< i \le N$ and $N+1\le j<\gamma$. 
This is treated as in case (D)  of Section \ref{ssec:neither} (see \eqref{eq:601}), 
which results (after $\rho$-removal and oxbow removal) in some $\nuf_1\in\sigma(\NST(\Gm))$
including $z$. Either $\nuf_1$ includes some facial site or it does not, and in either case the claim is shown.

\item Assume $z \hsim \nuf'_\rho(u-)$ but $z \nhsim\nuf'_\lam(w+)$.
Find the earliest  $x\in\nuf'_\rho(u-)$ satisfying $x\hsim z$ (noting that $x\in V$); 
truncate $\nuf'_\rho(u-)$
at $x$ to the subpath $\nuf'(x-)$, and add the vertices $z,y_{N+1}, y_{N+2},\dots,y_\gamma$ to $\nuf_\lam'(w-)$. Let $J$ be the face such that $z, x \in\ol J$; if $z\ne \phi(J)$ and $z\nsim x$ in $G$,
we add $\phi(J)$ also.
After $\rho$-removal and oxbow removal, one obtains the required $\nuf_1$. 
It needs be checked that
\begin{equation}\label{eq:611}
\nuf'(x-) \nhsim \{y_{N+1},y_{N+2},\dots,y_\gamma\},
\end{equation}
and this holds in a similar manner to that of case (A) of Section \ref{ssec:neither}.
A similar argument holds with $u$ and $w$ interchanged.

\item Assume $z \hsim\nuf'_\rho(u-)$ and $z\hsim \nuf'_\lam(w+)$.
By \eqref{eq:610}, $z\in V$.
Find the earliest $x \in \nuf'_\rho(u-)$ satisfying $x\hsim z$, and the latest 
$y \in\nuf'_\lam(w+)$ satisfying $y\hsim z$; truncate the two paths at $x$ and $y$ respectively, and add the
vertex $z$ and any required facial site. The outcome is the required $\nuf_1$.

\end{numlist}

\subsection{(ii) Assume that $\nuf\cap Q$ is the edge $\langle v,w\rangle$,}\label{ssec:8.2}
where $w=y'$, 
and consider cases (E), (F) of 
Section \ref{ssec:either}.
In (E), we may take $\nuf_1$ to be as defined there.
Consider the second case (F)  as illustrated on the right of Figure \ref{fig:basic''}.
We follow Section \ref{ssec:8.1} above but with differences as follows.

\begin{numlist}
\item
If $z \nhsim \nuf'_\rho(u-) \cup \nuf(w+)\sm\{w\}$,
we add to $\nuf'_\rho(u-)\cup\nuf(w+)$ the vertex sequence 
$y_\alpha,y_{\alpha+1},\dots, y_N(=y), z$. If
\begin{equation}\label{eq:621}
\{y_\alpha,y_{\alpha+1},\dots, y_N\} \nhsim \nuf(w+),
\end{equation}
the resulting path $\nuf_1$ (after $\phi$-removal and oxbow removal) is as required.
If \eqref{eq:621} fails, we find the earliest $I$ such that $\alpha \le I\le N$ and $y_I\hsim\nuf(w+)$
and the latest $x\in \nuf(w-)$ such that $y_I\hsim x$. Note that $y_I,x\in V$, so that they lie in some common
cycle $J$. Now apply $\phi$-removal and oxbow removal to
the walk $(\nuf_\rho(u-), y_\alpha,\dots, y_I, \phi(J),\nuf(x-))$ 
to obtain  $\nuf_1\in\sigma(\NST(\Gm))$ that includes a facial site.

\item Assume $z \hsim\nuf'_\rho(u-)$ but $z\nhsim \nuf(w+)\sm\{w\}$.
Find the earliest $x\in \nuf'_\rho(u-)$ satisfying $x\hsim z$
(noting that $x\in V$); truncate $\nuf'_\rho(u-)$
at $x$, and add $z$ to $\nuf(w+)$ (and also the facial site $\phi(J)$ if needed, as explained above),
to obtain the required $\nuf_1$. We argue similarly with $u$ and $w$ interchanged.

\item Assume $z \hsim\nuf'_\rho(u-)$ and $z \hsim \nuf(w+)\sm\{w\}$.
Find the earliest $x \in \nuf'_\rho(u-)$ satisfying $x\hsim z$, and the latest 
$y \in \nuf(w-)$ satisfying $y\hsim z$; truncate the two paths at $x$ and $y$ respectively, and add the
vertex $z$ (possibly with facial sites as needed). The outcome is the required $\nuf_1$.

\end{numlist}

\subsection{(iv) Assume that $\nuf\cap Q$ comprises the two edges $\langle v,w\rangle$, 
$\langle v,y \rangle$,}\label{ssec:8.3}
so that $u=y$ and $w=y'$. The idea is to replace $v$ by $z$, and we proceed as above.

\begin{numlist}
\item
If $z \nhsim (\nuf'(u-)\sm\{u\}) \cup (\nuf'(w+)\sm\{w\})$,
we remove $v$ from $\nuf$ and add $z$ to $\nuf'(u-)\cup\nuf(w+)$.

\item Assume $z \hsim(\nuf'(u-)\sm \{u\})$ but $z\nhsim (\nuf'(w+)\sm\{w\})$.
Find the earliest $x\in \nuf'(u-)$ satisfying $x\hsim z$
(noting that $x\in V$); truncate $\nuf'(u-)$
at $x$, and add $z$ to $\nuf(w+)$ (and also the facial site $\phi(J)$ if needed, as explained above),
to obtain the required $\nuf_1$.
We argue similarly with $u$ and $w$ interchanged.

\item Assume $z \hsim(\nuf'(u-)\sm\{u\})$ and $z \hsim (\nuf'(w+)\sm\{w\})$.
Find the earliest $x \in \nuf'(u-)$ satisfying $x\hsim z$, and the latest 
$y \in \nuf'(w-)$ satisfying $y\hsim z$; truncate the two paths at $x$ and $y$ respectively, and add the
vertex $z$ (possibly with facial sites as needed). The outcome is the required $\nuf_1$.

\end{numlist}

\section*{Acknowledgement}

The author acknowledges discussions with Zhongyang Li concerning the problem addressed here.

\providecommand{\bysame}{\leavevmode\hbox to3em{\hrulefill}\thinspace}
\providecommand{\MR}{\relax\ifhmode\unskip\space\fi MR }
\providecommand{\MRhref}[2]{%
  \href{http://www.ams.org/mathscinet-getitem?mr=#1}{#2}
}
\providecommand{\href}[2]{#2}


\begin{thebibliography}{10}

\bibitem{AG}
M.~Aizenman and G.~R. Grimmett, \emph{Strict monotonicity for critical points
  in percolation and ferromagnetic models}, J. Statist. Phys. \textbf{63}
  (1991), 817--835.

\bibitem{ACM}
R.~Ayala, M.~J. Ch\'{a}vez, A.~M\'{a}rquez, and A.~Quintero, \emph{On the
  connectivity of infinite graphs and {$2$}-complexes}, Discrete Math.
  \textbf{194} (1999), 13--37.

\bibitem{BBR}
P.~Balister, B.~Bollob\'as, and O.~Riordan, \emph{Essential enhancements
  revisited},  (2014), \url{http://arxiv.org/abs/1402.0834}.

\bibitem{BKK}
T.~Biedl, G.~Kant, and M.~Kaufmann, \emph{On triangulating planar graphs under
  the four-connectivity constraint}, Algorithmica \textbf{19} (1997), 427--446.

\bibitem{BIS}
C.~P. Bonnington, W.~Imrich, and N.~Seifter, \emph{Geodesics in transitive
  graphs}, J. Combin. Th. Ser. B \textbf{67} (1996), 12--33.

\bibitem{CFKP}
J.~W. Cannon, W.~J. Floyd, R.~Kenyon, and W.~R. Parry, \emph{Hyperbolic
  geometry}, Flavors of Geometry (S.~Levy, ed.), MSRI Publications No.\ 31,
  Cambridge Univ. Press, Cambridge, 1997, pp.~59--115.

\bibitem{Dirac63}
G.~A. Dirac, \emph{Extensions of {M}enger's theorem}, J. London Math. Soc.
  \textbf{38} (1963), 148--161.

\bibitem{GP}
G.~R. Grimmett, \emph{Percolation}, 2nd ed., Springer, Berlin, 1999.

\bibitem{G-pgs}
\bysame, \emph{{Probability on Graphs}}, 2nd ed., Cambridge University Press,
  Cambridge, 2018.

\bibitem{GL-match}
G.~R. Grimmett and Z.~Li, \emph{Percolation critical probabilities of matching
  lattice-pairs}, Random Struct. Alg. \textbf{65} (2024), 832--856.

\bibitem{GL21a}
\bysame, \emph{Hyperbolic site percolation}, Random Struct. Alg. \textbf{66}
  (2025), paper 21262.

\bibitem{HP}
O.~H\"aggstr\"om and Y.~Peres, \emph{Monotonicity of uniqueness for percolation
  on {C}ayley graphs: All infinite clusters are born simultaneously}, Probab.
  Th. Rel. Fields \textbf{113} (1999), 273--285.

\bibitem{halin70}
R.~Halin, \emph{Die {M}aximalzahl fremder zweiseitig unendlicher {W}ege in
  {G}raphen}, Math. Nachr. \textbf{44} (1970), 119--127.

\bibitem{jmh61}
J.~M. Hammersley, \emph{Comparison of atom and bond percolation processes}, J.
  Math. Phys. \textbf{2} (1961), 728--733.

\bibitem{Jung}
H.-O. Jung, \emph{An extension of {W}hitney's theorem to infinite strong
  triangulations}, Abh. Math. Sem. Univ. Hamburg \textbf{64} (1994), 131--139.

\bibitem{K82}
H.~Kesten, \emph{{Percolation Theory for Mathematicians}}, Birkh\"auser,
  Boston, 1982,
  \url{https://www.statslab.cam.ac.uk/~grg/books/kesten/kesten-book}.

\bibitem{KU}
K.~Knauer and T.~Ueckerdt, \emph{Decomposing 4-connected planar triangulations
  into two trees and one path}, J. Combin. Theory Ser. B \textbf{134} (2019),
  88--109.

\bibitem{Kr}
B.~Kr\"on, \emph{Infinite faces and ends of almost transitive plane graphs},
  Hamburger Beitr\"age zur Mathematik \textbf{257} (2006), 22 pp,
  \url{https://preprint.math.uni-hamburg.de/public/hbm.html}.

\bibitem{LP}
R.~Lyons and Y.~Peres, \emph{{Probability on Trees and Networks}}, Cambridge
  University Press, Cambridge, 2016,
  \url{https://rdlyons.pages.iu.edu/prbtree/}.

\bibitem{SMartin}
S.~Martineau, \emph{Locally infinite graphs and symmetries}, Grad. J. Math.
  \textbf{2} (2017), 42--50.

\bibitem{Moh}
B.~Mohar, \emph{Embeddings of infinite graphs}, J. Combin. Th. Ser. B
  \textbf{44} (1988), 29--43.

\bibitem{SR99}
R.~H. Schonmann, \emph{Stability of infinite clusters in supercritical
  percolation}, Probab. Th. Rel. Fields \textbf{113} (1999), 287--300.

\bibitem{SE63}
M.~F. Sykes and J.~W. Essam, \emph{Exact critical percolation probabilities for
  bond and site problems in two dimensions}, Phys. Rev. Lett. \textbf{10}
  (1963), 3--4.

\bibitem{Wat}
M.~E. Watkins, \emph{Infinite paths that contain only shortest paths}, J. Comb.
  Th. B \textbf{41} (1986), 341--355.

\end{thebibliography}
\end{document}